\documentclass[11pt,letterpaper]{article}
\pdftrailerid{}
\usepackage[margin=1.05in]{geometry}
\usepackage{amsmath,amssymb,amsthm,booktabs,longtable}
\usepackage[hidelinks]{hyperref}
\usepackage[T1]{fontenc}
\usepackage{lmodern}
\usepackage[mathlines]{lineno}
\hypersetup{pdftitle={Minimum central circles: an effective characterization of the global asymptotic constant},
 pdfauthor={Maurizio Falconi}}
\newtheorem{theorem}{Theorem}
\newtheorem{lemma}[theorem]{Lemma}
\newtheorem{proposition}[theorem]{Proposition}

\theoremstyle{definition}\newtheorem{definition}[theorem]{Definition}
\theoremstyle{remark}
\newcommand{\Rstar}{R^\ast}
\newcommand{\ch}{R_{\mathrm{chain}}}
\newcommand{\full}{R_{\mathrm{full}}}
\newcommand{\CC}{C_\ast}
\DeclareMathOperator{\atan}{arctan}
\DeclareMathOperator{\asin}{arcsin}
\title{Minimum central circles:\\an effective characterization of the global asymptotic constant}
\author{Maurizio Falconi}
\date{3 November 2026}
\begin{document}
\maketitle
\nolinenumbers
\begin{abstract}
Let $\Rstar(n)$ be the minimum radius of a central circle externally tangent
to nonoverlapping circles of radii $1,\ldots,n$. We prove that
$\Rstar(n)=\CC n^2+o(n^2)$ and give finite linear programs that enclose
$\CC$ with prescribed error.

An exact angular comparison reduces the problem
to marked points on a line with separation at least the geometric mean of
each pair of marks. Concatenation with bounded joining cost and reassignment
to the prescribed normalized radii prove existence of the limit. Balanced finite-word
programs give a bracket of width $(1/k+1/r)/\pi$, before directed arithmetic
error, using $k$ types and words of length $r$.

We also obtain explicit geometric bounds. For the upper bound, we recover
finite and countable reflected-block constructions by actual permutations,
retaining all pair constraints. A coupled terminal-subset argument gives
an explicit lower bound strictly above $1/8$. Together these give
$C_{\rm term}+\eta_{\rm width}\le\CC\le U_4$.
The lower endpoint comes from the terminal-subset argument, and the upper
endpoint from the four-block construction; neither is asserted sharp.

The finite-word scheme establishes computability; efficient high-precision
evaluation and the structure of global optimizers remain open.
\end{abstract}

\section{Problem and Relation to Prior Finite Work}

We seek the smallest central circle to which all the prescribed outer
circles can be externally tangent without overlapping. The outer radii
are consecutive positive integers starting at one; outer circles need not
be tangent to each other. For a cyclic order $\sigma$, the adjacent
chain radius $\ch(\sigma)$ ignores nonadjacent constraints, whereas
$\full(\sigma)$ is the least radius satisfying every pair constraint in
that order. The global quantity is $\Rstar(n)=\min_\sigma\full(\sigma)$.
Always $\ch(\sigma)\le\full(\sigma)$. A feasible construction supplies
an upper bound, not an optimum certificate.

The question here is the leading growth of this global minimum as the
number of outer circles increases. We prove existence of the limiting
coefficient and characterize it by finite linear programs with a prescribed
error.

Paper I~\cite{v3} develops the angular model, the anti-Monge/Supnick
chain-ordering connection and finite optimization. Its reported finite
certification range is $3\le n\le14$; these finite computations are not
premises of the present asymptotic proofs. The original version~\cite{v1}
proposed the coefficient $1/8$, and the corrective version~\cite{v2}
records its disproof by the sequel.
No finite table or global floating-circle classification is reproduced here.

The scientific results here appeared in the first version of the standalone
preprint arXiv:\allowbreak2609.13630~\cite{preprint}. This revision supplies autonomous
proof supplements and makes their dependencies explicit. Its contribution
over that version is proof exposition and reproducibility, not a new
asymptotic coefficient.

Sections~\ref{sec:line}--\ref{sec:lp} develop the proof of the effective
characterization. The argument passes from the angular model to an all-pair
line problem, then uses genuine-label concatenation and balanced finite-word
programs. The lower and upper bounds supplied by these finite-word programs
converge to the same constant, with the explicit error bound stated in
Theorem~\ref{thm:lp}. Separate explicit
constructions give the interval stated in the abstract and disprove the
coefficient conjecture of \cite{v1}. Their complete proof dependencies are
given in the supplementary material; the main variational characterization does
not assume either endpoint mechanism is optimal.

The line problem has a familiar geometric interpretation. Disks of radii
$a_i$ tangent to the horizontal axis, with centers $(2x_i,a_i)$, are
nonoverlapping precisely when $|x_i-x_j|\ge\sqrt{a_i a_j}$.
For $a_i\le1$, their horizontal envelope span differs from
$2(\max_i x_i-\min_i x_i)$ by at most two. Alt et al.~\cite{shelf}
study minimum-width shelf packing and give approximation, special-case exact
algorithms and hardness results. The shelf model and its geometric-mean
separation rule are pre-existing. Our claim concerns the uniform-label
asymptotic regime and its transfer to central-circle geometry, with a
balanced-word error bound. General shelf-packing hardness does not imply
hardness of evaluating this particular constant.

Another related setting is weighted disk-contact recognition.
Klemz et al.~\cite{contact} study prescribed disk radii
and contact graphs, including stars with arbitrary weights. Ringmin instead
fixes the outer radii to $1,\ldots,n$, varies the central radius, and permits
additional outer contacts. We make no transfer of their recognition
hardness to evaluation of $\CC$.

Likewise, difference-constraint feasibility is classical~\cite{temporal};
its use for a fixed Ringmin order is already part of Paper I. We do not
claim a new general method for such systems. The contribution of the
reflected-block theorem is a special all-pairs reduction and quantitative
recovery by genuine labels. The comparison here is scoped to these directly
related sources, not an exhaustive priority survey.

\section{Exact Angular Comparison}\label{sec:line}

The cosine law gives the required pair angle
\begin{equation}\label{eq:theta}
 \theta_R(u,v)=2\asin\sqrt{\frac{uv}{(R+u)(R+v)}}
 =2\atan\sqrt{\frac{uv}{R(R+u+v)}}.
\end{equation}
Both directed cyclic arcs between a pair must be at least this angle.
This is equivalent to non-overlap, since their smaller separation belongs
to $[0,\pi]$. The angle is positive, symmetric, increasing in either
outer radius and strictly decreasing in $R>0$.
For a cyclic order $\sigma=(u_1,\ldots,u_n)$, $n\ge3$, with $u_{n+1}=u_1$,
$\ch(\sigma)$ is the unique positive root of
$\sum_{i=1}^n\theta_R(u_i,u_{i+1})=2\pi$.
Indeed, the sum decreases continuously from $n\pi$ to zero. Summing the
adjacent constraints in any full placement proves $\ch(\sigma)\le\full(\sigma)$.

We next introduce an auxiliary problem on a line. Its comparison with the
circle problem will follow after a bound on the cost of closing a line
placement.

\begin{definition}
For a finite multiset $a=(a_1,\ldots,a_N)$ of marks in $[0,1]$, let $b(a)$
be the least span of a line placement satisfying
\begin{equation}\label{eq:line}
 |x_i-x_j|\ge\sqrt{a_i a_j}\quad(i\ne j).
\end{equation}
Distinct zero-mark points may coincide. Write $b_n=b(1/n,2/n,\ldots,1)$.
\end{definition}

For a fixed word $w=(w_1,\ldots,w_N)$ in left-to-right order, its minimum
span $\ell(w)$ is attained by
\begin{equation}\label{eq:recurrence}
 x_1=0,\qquad x_i=\max_{j<i}\{x_j+\sqrt{w_iw_j}\},\qquad \ell(w)=x_N.
\end{equation}
Induction proves that these coordinates are componentwise no larger than
any feasible ordered coordinates with $x_1=0$; they also satisfy every
constraint. Nonnegative weights ensure ordered coordinates. Equivalently,
$\ell(w)$ is the longest increasing-index path from $1$ to $N$, with
edge weight $\sqrt{w_iw_j}$. Thus $b(a)$ is an attained minimum over a
finite set of orders. Deleting or decreasing marks cannot increase it;
scaling all marks by $h$ scales the minimum span by $h$.

\begin{lemma}[Bounded closing cost]\label{lem:closing}
Let $B(a)$ be the infimum circumference with both pair arcs at least
$\sqrt{a_i a_j}$. Then $b(a)\le B(a)\le b(a)+1$.
\end{lemma}
\begin{proof}
Cut a feasible circle to obtain a feasible line placement. Conversely,
close a feasible line placement by an additional gap of length one. One
arc for every pair is its line separation; the other contains the added
gap and is at least $1\ge\sqrt{a_i a_j}$. Both arcs are controlled.
\end{proof}

\begin{lemma}[Exact geometric squeeze]\label{lem:squeeze}
For every $n\ge3$,
\begin{equation}\label{eq:squeeze}
 \frac{b_n}{\pi n}-\frac1n
 \le\frac{\Rstar(n)}{n^2}\le\frac{b_n+1}{\pi n}.
\end{equation}
\end{lemma}
\begin{proof}
For $1\le u,v\le n$, the two expressions in \eqref{eq:theta} give
\[
 \frac{2\sqrt{uv}}{R+n}\le\theta_R(u,v)\le\frac{2\sqrt{uv}}R.
\]
The first uses $\asin z\ge z$ and
$\sqrt{(R+u)(R+v)}\le R+n$; the second uses $\atan z\le z$.
Scaling a feasible ring's angles by $(R+n)/(2n)$ gives a feasible linear
circle of circumference $\pi(R+n)/n$, proving the lower inequality.
Conversely, the circumference $b_n+1$ from Lemma~\ref{lem:closing},
scaled to $2\pi$, satisfies every original pair angle at
$R=n(b_n+1)/\pi$. This proves the upper inequality. No Taylor expansion,
adjacency reduction or assumption on the optimal order was used.
\end{proof}

\section{Existence by Genuine-Label Concatenation}

The closing estimate makes concatenation useful: copies of a finite line
placement can be joined at bounded cost. The next theorem combines this
construction with reassignment to the prescribed normalized radii.

\begin{theorem}[The leading limit exists]\label{thm:existence}
There is a constant $\CC$ such that
\begin{equation}\label{eq:limit}
 \Rstar(n)=\CC n^2+o(n^2),\qquad
 \pi\CC=E:=\lim_{n\to\infty}\frac{b_n}{n}
 =\inf_{n\ge2}\frac{b_n+1}{n}.
\end{equation}
\end{theorem}
\begin{proof}
Copies of a feasible line template, separated by unit gaps, satisfy all
cross-copy constraints as well as the internal ones. Fix $n\ge2$. Given
$N\ge n$, let $q=\lceil N/n\rceil$ and concatenate $q$ copies of an
optimal $b_n$ template. Its span is $q(b_n+1)-1$.
Retain the $N$ smallest marks. The $i$th sorted retained mark is
$\lceil i/q\rceil/n$. Scale all marks and coordinates by $h=nq/N$.
The resulting $i$th mark is at least $i/N$, so decreasing it to $i/N$
preserves every constraint. Each target label is assigned exactly once
to an actual retained point, with ties broken arbitrarily. Hence
\begin{equation}\label{eq:copies}
 b_N\le\frac{nq}{N}\bigl[q(b_n+1)-1\bigr].
\end{equation}
Intermediate marks exceeding one cause no problem: the line inequalities
are homogeneous, and final marks are exactly $1/N,\ldots,1$.

As $N\to\infty$, $q/N\to1/n$, giving
$\limsup_N b_N/N\le(b_n+1)/n$ for every fixed $n$.
The sequence is bounded between zero and one by equally spaced placements.
Taking $n$ along a subsequence attaining its liminf gives equality of
limsup and liminf. Taking the infimum and then $n\to\infty$ proves
the last equality in \eqref{eq:limit}. Lemma~\ref{lem:squeeze} transfers
the limit to the original ring problem for all sizes and both parities.
Finally, at least $n/2$ marks are at least $1/2$. Their successive line
separations are at least $1/2$, so $E\ge1/4$ and $\CC>0$.
\end{proof}

\section{Effective Finite-Word Characterization}\label{sec:lp}

We now turn from existence of the limit to approximation with a prescribed
error. There are two steps: replace the radii by finitely many types, then
replace long orderings by distributions of finite words. Each step contributes
one term to the bracket width.

\subsection{Finite Types and a Uniform Quantization Error}
For $k\ge1$, let $A_{k,q}$ contain $q$ copies of each mark
$1/k,\ldots,1$, and let $A^-_{k,q}$ instead contain $q$ copies of
$0,1/k,\ldots,(k-1)/k$. Define
$e_k=\lim_{q\to\infty}b(A_{k,q})/(kq)$ and
$e^-_k=\lim_{q\to\infty}b(A^-_{k,q})/(kq)$; these limits exist.
Indeed $a_q=b(A_{k,q})+1$ is subadditive by unit-gap concatenation.
Writing $q=hp+s$ at fixed $p$ proves
$\limsup a_q/q\le a_p/p$; then taking $p$ along a liminf subsequence
proves the limit. The same argument applies to the lower types.

The two multisets share every nonzero type except the top type. Zero marks
cost nothing. Adding the $q$ unit marks as a separate equally spaced
block costs at most $q$, including the join. Thus, also for $k=1$,
\[
 e^-_k\le e_k\le e^-_k+\frac1k.
\]
For $n=kq$, the sorted uniform marks $i/n$ lie between the two sorted
type multisets. Theorem~\ref{thm:existence} and monotonicity imply
\begin{equation}\label{eq:types}
 e_k-\frac1k\le E\le e_k.
\end{equation}

\subsection{Balanced Words and the Quantified Boundary Relaxation}
A word records a short consecutive piece of an ordering. The linear program
balances type counts across a distribution of words; the proof below controls
the cost of joining them.

Fix $k\ge1$ and $r\ge2$, and use all $k^r$ words
$w\in\{1/k,\ldots,1\}^r$, including repetitions. Let $c_j(w)$ count
occurrences of $j/k$. Compute $\ell(w)$ by \eqref{eq:recurrence}, and define
\begin{equation}\label{eq:lp}
\begin{split}
 \lambda_{k,r}=\min_p\ &\frac1r\sum_w p_w\ell(w),\\
 p_w\ge0,\quad &\sum_w p_w=1,\quad
 \sum_w p_wc_j(w)=\frac rk\quad(1\le j\le k).
\end{split}
\end{equation}
The polytope is nonempty: equally mix the $k$ constant words.

\begin{theorem}[Matching effective variational bounds]\label{thm:lp}
For every $k\ge1,r\ge2$,
\begin{equation}\label{eq:effective}
 \boxed{\quad\frac{\lambda_{k,r}-1/k}{\pi}\le\CC
 \le\frac{\lambda_{k,r}+1/r}{\pi}\quad}.
\end{equation}
Consequently $\CC=\pi^{-1}\lim_{k,r\to\infty}\lambda_{k,r}$, jointly,
with certified bracket width $(1/k+1/r)/\pi$.
\end{theorem}
\begin{proof}
Partition an optimal ordering of $A_{k,q}$ into consecutive blocks of
$r$ vertices, discarding fewer than $r$ leftover vertices. Their spans
sum to at most the full span and each is at least its word cost.
Empirical word distributions have convergent subsequences in the finite
simplex. Their limiting mean type counts are $r/k$, since the discarded
counts vanish after normalization. Hence $\lambda_{k,r}\le e_k$.

Conversely, the rational feasible polytope in \eqref{eq:lp} has a minimizing
rational vertex even though its cost entries are algebraic. Choose a
common denominator $D$. A batch with $Dp_w$ copies of each word has
exactly $Dr/k$ copies of each type, an integer by the count constraints.
Place every word at its minimum span and separate blocks by unit gaps.
All within-block and cross-block pairs are feasible. Repeating batches
gives asymptotic mean span
$\sum_w p_w(\ell(w)+1)/r=\lambda_{k,r}+1/r$.
The finite-type limit therefore satisfies
$e_k\le\lambda_{k,r}+1/r$. Combine with \eqref{eq:types}.
\end{proof}

\subsection{Exact Certificates and Computational Meaning}
Directed rational square-root enclosures of width $\epsilon$ give
word-span enclosures of width at most $(r-1)\epsilon$: every path in
\eqref{eq:recurrence} has at most $r-1$ edges. Rational endpoint LPs
then enclose $\lambda_{k,r}$ with additional width at most
$(r-1)\epsilon/r$. They can be solved by finite exact vertex enumeration.
Practical solvers may instead find solutions for subsequent exact verification.
For a rational lower cost $\ell^-(w)$, dual coefficients satisfying
\[
 y_0+\sum_j y_jc_j(w)\le\ell^-(w)/r\quad\hbox{for every word}
\]
certify the lower value $y_0+(r/k)\sum_jy_j$. Feasible rational $p$
certifies an upper value using upper costs. All such checks are finite
rational comparisons; equality is not decided by indefinite interval
refinement. Together with a rational enclosure of $\pi$, this proves
arbitrary-precision computability of $\CC$ in principle.

The word count is exponential in $r$. No efficient high-precision
algorithm or closed-form constant follows. The supplementary implementation
uses SciPy to find sparse primal and dual solutions; exact rational
arithmetic validates every equality and word inequality. Small word programs
need not improve the stronger explicit endpoint bounds below.

\section{Explicit Geometric Upper Constructions}\label{sec:upper}

The finite-word characterization gives a convergent scheme for the constant.
For explicit upper bounds, we use reflected-block constructions. The key step
is to recover actual permutations satisfying every pair constraint from the
continuous description.

Fix $0\le\alpha<1$ and nonnegative adjacent slab lengths
$\ell_1,\ldots,\ell_k$ of total $T<1-\alpha$. Put $A=1+\alpha$,
$t_0=0$, $t_j=\sum_{h\le j}\ell_h$. On each slab
$[t_{j-1},t_j]$, use equal masses at
\[
 (t,A+t,A+t_{j-1}+t_j-t),\qquad
 (t,A+t_{j-1}+t_j-t,A+t).
\]
After $T$, use the diagonal $(t,1+\{t+\alpha\},1+\{t+\alpha\})$.
This defines a measure $\mu$ on $[0,1]\times[1,2]^2$ with uniform low
and separate high marginals. Its complete cell cost is
\begin{equation}\label{eq:blockcost}
 C(\mu)=\frac1{4\pi}\int
 \max\{\sqrt t(\sqrt X+\sqrt Y),\sqrt{XY}\}\,d\mu.
\end{equation}
Correct marginals alone do not establish finite recovery.

\begin{theorem}[General block transfer]\label{thm:blocks}
Every such finite partition has genuine alternating-half permutations with
all-pairs-feasible radii $\rho_m=\full(\sigma_m)$ and
\[
 \left|\frac{\rho_m}{(2m)^2}-C(\mu)\right|
 \le\frac{(6k^2+28k+1060)/m+16384/(3m^2)}{4\pi}
\]
for $m\ge\max\{2048,\lceil2/(1-\alpha-T)\rceil\}$.
It follows that $\CC\le C(\mu)$. Countably many slabs also transfer
if their total length is strictly below $1-\alpha$.
\end{theorem}
\begin{proof}[Proof with explicit construction]
Let $s=\lfloor\alpha m\rfloor$, $l_j=2\lfloor\ell_jm/2\rfloor$,
$a_0=0$ and $a_j=\sum_{h\le j}l_h$. Reverse only the even ranks within
each $(a_{j-1},a_j]$ to obtain a permutation $J$ of $\{1,\ldots,m\}$.
Set $J(i)=i$ on odd ranks and on every index outside these blocks.
Set $P_i=m+1+((J(i)+s-1)\bmod m)$ and
$\sigma_m=(1,P_1,\ldots,m,P_m)$, with $P_0=P_m$.
Disjoint involutions and the cyclic shift prove bijectivity, including
zero blocks and length-two identities. The lower bound on $m$ keeps the final
block endpoint at least two ranks before the high wrap, and its exit
strictly before that wrap.

Double panels of width $2/m$, allocated half to each parity, recover
the two reflected orientations. Interior coordinate error is at most
$4/m$. There are at most $k+3$ exceptional actual cells: shared exits,
the cyclic predecessor, and the two wrap cells. Rounded slab boundaries
and wrap differ on a set of length at most $[k(k+1)+1]/m$, with remaining
coordinate error at most $(4k-1)/m$. Write
$g(t,X,Y)=\max\{\sqrt t(\sqrt X+\sqrt Y),\sqrt{XY}\}$ for the integrand
in \eqref{eq:blockcost}. It is $4$-Lipschitz and lies in $[1,3)$.
The actual normalized full-cell moment is
\[
 G_m=\frac1m\sum_{i=1}^m g(i/m,P_{i-1}/m,P_i/m).
\]
It differs from
$\int g\,d\mu$ by at most $(6k^2+28k+36)/m$.

For arbitrary distinct highs in $[m+1,2m]$, the exact full-feasibility
criterion for this alternating order is
\[
 \sum_{i=1}^m\max\{\theta_R(P_{i-1},i)+\theta_R(i,P_i),
                         \theta_R(P_{i-1},P_i)\}\le2\pi.
\]
Its sufficiency follows from the high-shell triangle inequality
$2\theta_R(m+1,m+1)>\theta_R(2m,2m)$ and high/low separation:
each directed high path contracts through whole cells, and paths with
low endpoints either contain such a high path or an incident high gap
that already exceeds the endpoint requirement. Both directions are
treated separately. At its unique root, choose the first half of each
cell at its required low/high angle and put any excess in the other half.
This closes at $2\pi$ and satisfies all pairs.

At $R=4cm^2$, $c\ge1/32$, the full score differs from $G_m/(2c)$ by
at most $1024/m+16384/(3m^2)$. For $m\ge2048$ this first brackets
the actual normalized root in $(1/32,1/2)$, then yields the stated error.
Deleting label $2m$ gives the odd upper bound. For countably many slabs,
truncate and use that the omitted complete-cost integral is at most six
times the omitted length. A diagonal sequence with $m$ growing faster
than $k^3$ pays both finite and tail errors. Full cell/seam formulas and
the angle estimate are supplied in the versioned standalone upper supplement S2~\cite{upper}.
\end{proof}

For the explicit four-block construction, let $x_*$ be the unique minimizer
on $[0,1]$ of
\[
 E(x)=\int_0^x\!\left[
 \max\{\sqrt{(1+t)(1+x-t)},\sqrt t(\sqrt{1+t}+\sqrt{1+x-t})\}
 -\max\{1+t,2\sqrt{t(1+t)}\}\right]dt.
\]
Let $K(\alpha)=\int_0^1\max\{\sqrt{t h_\alpha(t)},h_\alpha(t)/2\}\,dt$,
where $h_\alpha(t)=1+\{t+\alpha\}$, and let $\widehat\alpha$ be the
unique zero of $K'(\alpha)+(1+\alpha)E(x_*)$ on $[0,1/2]$.
Set $A=1+\widehat\alpha$ and $\lambda=Ax_*$.
The second width $\epsilon_b$ is the unique mixed-branch stationary
minimum for a separate reflection starting at $\lambda$; the third width
$\Delta_*$ is the analogous unique minimum starting at
$v=\lambda+\epsilon_b$. Each width is minimized over
$[0,A/3-u]$ at its fixed start $u$, retaining the full max.
The supplement \cite{upper} gives their exact defining integrals,
existence/uniqueness proofs and rational isolations:
\[
\begin{gathered}
 .1093<\widehat\alpha<.10931,\qquad .2876<x_*<.2877,\\
 .043<\epsilon_b<7/160,\qquad 29/5000<\Delta_*<27/4000.
\end{gathered}
\]
These terminating decimals are rational bracket endpoints, not definitions
or numerical replacements for the parameters.

Let $C_3(d)$ be \eqref{eq:blockcost} for the three lengths
$(\lambda,\epsilon_b,d)$, and $C_4(\eta)$ the same cost for the four
lengths $(\lambda,\epsilon_b,\Delta_*,\eta)$.
Use $\eta=1/20000$ and define $U_4=C_4(1/20000)$.
If $w=\lambda+\epsilon_b+\Delta_*$, $B=A+w$ and
$M=B+\eta/2$, the fourth increment at $0<\eta\le1/20000$ is exactly
\[
 C_4(\eta)-C_3(\Delta_*)=
 -\frac1{4\pi}\int_{-\eta/2}^{\eta/2}
       \frac{z^2}{M+\sqrt{M^2-z^2}}\,dz.
\]
Rational input margins keep both changed maxima strictly chord.
Since $2B\le M+\sqrt{M^2-z^2}\le2M$, $M<3/2$ and $\pi<4$,
\begin{equation}\label{eq:U4}
 \CC\le U_4<C_3(\Delta_*)-\frac1{4608000000000000}.
\end{equation}
The leading decrement is $-\eta^3/(96\pi B)$, with remainder between
zero and $\eta^4/(192\pi B^2)$. The theorem transfers this literal
continuous saving to genuine full geometry. It does not make four blocks
globally optimal or optimize an infinite partition.

\section{An Explicit Global Lower Endpoint}

Define $\tau$ as the unique root of $\cos\tau=\tau$ in $(0,1)$, set
$q=(1-\sin\tau)/(1+\sin\tau)$ and
$C_{\rm term}=\tau(1+q)/(2\pi)$.
For the fixed cutoffs
\[
 (\beta_1,\beta_2,\beta_3)=(2091/10000,10907/50000,23/100)
\]
put
\[
 D_i=\int_q^{\beta_i}[1+q-x-2\sqrt{x(1+q-x)}]\,dx,
 \quad F(h)=\sum_i h_iD_i-8\sum_i h_i^3.
\]
Let $a=113/12500$, $b=593/50000$. On
$h_i\ge0$, $h_1+h_2\le a$, $h_2+h_3\le b$, define
\[
 \eta_{\rm width}=\max_h
 \frac{\max\{F(h),0\}}{\pi(16+432\sum_i h_i)}.
\]
This maximum is uniquely attained at $(a-z_*,z_*,b-z_*)$, where $z_*$
is the unique zero on $[0,a]$ of $G=N'L+432N$,
$N(z)=F(a-z,z,b-z)$ and $L(z)=16+432(a+b-z)$.

\begin{proposition}[Explicit global lower bound]\label{prop:lower}
The global constant satisfies $\CC\ge C_{\rm term}+\eta_{\rm width}$,
with exact directed enclosure
\[
 .14056946887766098063257
 < C_{\rm term}+\eta_{\rm width}
 < .14056946887766098063392.
\]
In particular $\CC>1/8$, so both the v1 coefficient conjecture and its
stronger $n^2/8-\Rstar(n)=O(\sqrt n)$ claim are false.
\end{proposition}
\begin{proof}[Proof outline]
The complete proof L1--L8 and exact arithmetic certificates are in the
unchanged standalone lower supplement S1~\cite{lower}.
Its essential coupled information is one shared crossing energy for the
same order and all three nested terminal restrictions. An explicit assignment
dual bounds the reflected-grid energy by eight times the excess chain cost.
The absolute signed deletion discrepancy at cutoff $i$ is at most $54$
times this energy plus twice its crossing term. Separated strips charge
the common energy once, yielding
the scalar gain $\max\{F(h),0\}/(16+432\sum h_i)$.
The exact angular squeeze transfers each induced chain bound to any original
full feasible configuration. Uniform finite errors vanish at every fixed
strictly separated width vector.

For the width optimum, $G'<0$ and rational endpoint bounds isolate its zero.
Set $u=(a-z_*,z_*,b-z_*)$, $\mathcal L(h)=16+432\sum_i h_i$ and
$\rho=F(u)/\mathcal L(u)$. The gradient of $F-\rho\mathcal L$ at $u$ is a positive
combination of the two constraint normals. Its exact remainder is
$-8\sum_i(h_i-u_i)^2(h_i+2u_i)$, proving global uniqueness on the whole
closed region. To pass to the geometric limit, first use $tu$ at fixed
$0<t<1$, take $n\to\infty$, then $t\uparrow1$. The boundary
widths fail finite floor separation infinitely often; the proof never uses
them as an eventual finite separation condition. The uniform finite error bounds for this
limiting argument are supplied in \cite{lower}. The result does not assert
that this particular lower mechanism reaches $\CC$.
\end{proof}

\section{Reproducibility and Open Questions}

The limit and finite-word proofs are complete in Sections~3--4. The
Proposition~7 supplement S1~\cite{lower} independently includes
limit existence for its own use. The upper supplement S2~\cite{upper}
gives the implicit-parameter definitions, their isolation gates, all-pairs
feasibility and geometric transfer.

The \href{https://github.com/falker47/ringmin/raw/4b3721df8061f1992cfd242df69b1a0643216d8f/paper_assets/paper_ii_journal/source_bundle.zip}{frozen companion ZIP}
contains the earlier source edition, the supplements and a claim-to-proof map.
S1 and its verifier are in \texttt{lower\_supplement/}. Run
\mbox{\texttt{python -I -S verify.py}} there. For S2, see
\texttt{upper\_supplement.pdf} and the checkers in \texttt{checks/}.
From the archive root, run \mbox{\texttt{python -I -S verify\_math.py}}
for the aggregate bounded arithmetic checks, including S1 and S2.

The package includes isolated build instructions, source hashes and bounded
checks for line placements, label reassignment and rational word-LP
certificates. SciPy is used only for LP candidate discovery; exact rational
checks validate the proposed primal and dual data. S1 and the upper endpoint
sign gates use Python's standard library. Negative controls test deliberately
omitted pair constraints, missing closing gaps and corrupt certificates.
These finite checks corroborate their declared arithmetic scopes; the
written proofs establish universal quantifiers and limiting statements.
The manifest in the ZIP identifies the archived source edition and
the supplement files used here.

The leading normalized behavior is characterized effectively, while neither
explicit endpoint is known to equal $\CC$. Practical evaluation or
simplification of the constant, sharp explicit constructions, microscopic
optimizer structure and subleading terms remain open. The number of words
in the finite programs is exponential in their length; computability does
not give an efficient high-precision algorithm. The limit theorem neither
certifies a finite optimum beyond $n=14$ nor establishes a universal
floating-circle cascade or a contact structure for global optimizers.

\section*{Acknowledgment}
AI assistance was used for proof exploration, software, drafting and
adversarial checks. Responsibility for the mathematical statements and
presentation rests with the author.

\end{document}


\maketitle
\begin{abstract}
We give a self-contained proof of the explicit lower bound
$C_*\ge\Cterm+\etaw$ for the asymptotic central-circle constant of
Ringmin. Three nested terminal deletions of one cyclic order share one
quadratic energy. An explicit dual potential, signed deletion estimate,
crossing lemma, scalar minimax and global cubic certificate yield the bound.
All constants and directed rational checks are specified here and in two
small included arithmetic scripts. In particular,
\[
0.14056946887766098063257<\Cterm+\etaw
<0.14056946887766098063392.
\]
The coefficient is a lower bound, not an identified geometric optimum.
This supplement reconstructs Proposition 7 of arXiv:2609.13630 and does
not revise that manuscript or assert external scientific acceptance.
\end{abstract}
{\small\tableofcontents}
\clearpage

\section{Statement, notation and proof dependencies}\label{sec:statement}
The Ringmin problem places circles of radii $1,\ldots,n$ externally
tangent to a central circle of radius $R$, requiring the interiors of
every pair of outer circles to be disjoint. Write $\Rstar(n)$ for the
infimum of such feasible $R$. For a cyclic order $\omega$ of at least
three distinct positive radii, $\Rc(\omega)$ is the adjacent angular
closure root defined in Section~\ref{sec:geometry}. It is only a necessary
bound for a full feasible placement of that order.

Let $\tau$ be the unique solution of $\cos\tau=\tau$ in $(0,1)$ and put
\begin{equation}\label{eq:parameters}
q=\frac{1-\sin\tau}{1+\sin\tau},\qquad
\Cterm=\frac{\tau(1+q)}{2\pi},\qquad
(\beta_1,\beta_2,\beta_3)=
\left(\frac{2091}{10000},\frac{10907}{50000},\frac{23}{100}\right).
\end{equation}
Define
\begin{align}
D_i&=\int_q^{\beta_i}\left[1+q-x-2\sqrt{x(1+q-x)}\right]\dd x,
\label{eq:Di}\\
a&=\beta_2-\beta_1=\frac{113}{12500},\qquad
b=\beta_3-\beta_2=\frac{593}{50000},\nonumber\\
\mathcal P&=\{h\in[0,\infty)^3:h_1+h_2\le a,\ h_2+h_3\le b\},\nonumber\\
F(h)&=\sum_{i=1}^3h_iD_i-8\sum_{i=1}^3h_i^3,\qquad
\mathcal L(h)=16+432\sum_{i=1}^3h_i,\label{eq:objective}\\
\eta(h)&=\frac{\pos{F(h)}}{\pi\mathcal L(h)},\qquad
\etaw=\max_{h\in\mathcal P}\eta(h),\quad r_+=\max\{r,0\}.\nonumber
\end{align}

\begin{theorem}\label{thm:main}
The limit $C_*=\lim_{n\to\infty}\Rstar(n)/n^2$ exists. The maximum
in \eqref{eq:objective} is positive and uniquely attained at
$u=(a-z_*,z_*,b-z_*)$, where $z_*$ is the unique root on $[0,a]$ of
the polynomial $G$ defined in \eqref{eq:G}. Moreover
\begin{equation}\label{eq:mainbound}
\begin{gathered}
C_*\ge\Cterm+\etaw,\\
\frac{14056946887766098063257}{10^{23}}
<\Cterm+\etaw<
\frac{14056946887766098063392}{10^{23}}.
\end{gathered}
\end{equation}
\end{theorem}

Every terminating decimal in this document denotes an exact rational.
We use $\alpha=k/n$ for the finite grid's lower endpoint; $a,b$ are
reserved for the two width constraints. The root $z_*$ belongs solely
to this lower bound and has no connection to an upper-bound parameter.
The proof is organized as follows.

\begin{center}\small
\begin{tabular}{@{}cll@{}}\toprule
Input & Complete argument & Dependency\\\midrule
L1 & Angular geometry and deletion (\S\ref{sec:geometry}) & definitions\\
L2 & Exact terminal parameters (\S\ref{sec:terminal}) & elementary calculus\\
L3 & Dual and common energy (\S\ref{sec:dual}) & L2 grid domain\\
L4 & Signed deletion and floors (\S\ref{sec:deletion}) & L3\\
L5 & Shared crossings (\S\ref{sec:crossings}) & common measure\\
L6 & Finite minimax (\S\ref{sec:minimax}) & L1, L3--L5\\
L7 & Width optimum and enclosure (\S\ref{sec:width}, \S\ref{sec:arithmetic}) & L2\\
L8 & Ordered limits (\S\ref{sec:limit}) & L6, L7; limit existence\\\bottomrule
\end{tabular}
\end{center}
The final limit-existence argument is included in Section~\ref{sec:existence}.
The references identify provenance; no external proof or operational log is
needed to follow the arguments below.

\section{Angular geometry and deletion (L1)}\label{sec:geometry}
For radii $u,v>0$ the angle required by non-overlap is
\begin{equation}\label{eq:kernel}
\theta_R(u,v)=2\arcsin\frac{\sqrt{uv}}{\sqrt{(R+u)(R+v)}}
=2\arctan\frac{\sqrt{uv}}{\sqrt{R(R+u+v)}}.
\end{equation}
Indeed the centers lie at distances $R+u,R+v$ from the origin. The cosine
law says that their distance is at least $u+v$ exactly when the smaller
angular separation is at least \eqref{eq:kernel}. The second identity
uses $1-uv/((R+u)(R+v))=R(R+u+v)/((R+u)(R+v))$.
Both directed angular arcs must therefore be at least $\theta_R(u,v)$.

\begin{lemma}[Uniform root sandwich]\label{lem:sandwich}
For a cycle $\omega$ of at least three radii in $\{1,\ldots,n\}$, put
$S(\omega)=\sum_{\{u,v\}\in\omega}\sqrt{uv}$, counting every cyclic edge
once, including the closing edge. Then the equation
$\sum_{\omega}\theta_R(u,v)=2\pi$ has a unique positive root and
\begin{equation}\label{eq:sandwich}
\frac{S(\omega)}\pi-n\le\Rc(\omega)\le\frac{S(\omega)}\pi.
\end{equation}
\end{lemma}
\begin{proof}
Each angle is continuous and strictly decreasing from $\pi$ to zero as
$R$ increases from zero to infinity. The sum starts above $2\pi$.
For $1\le u,v\le n$, $\arcsin t\ge t$, $\arctan t\le t$ and
\eqref{eq:kernel} give
\[
\frac{2\sqrt{uv}}{R+n}\le\theta_R(u,v)\le\frac{2\sqrt{uv}}R.
\]
At the root these give $S/(R+n)\le\pi\le S/R$, which is precisely
\eqref{eq:sandwich}. No scale assumption or parity restriction is used.
\end{proof}

\begin{lemma}[Deletion from full geometry]\label{lem:deletiongeometry}
Delete any outer circles from a full feasible configuration, retaining
at least three circles with their original radii. If $\omega'$ is the
cyclic order of the survivors, then $\Rc(\omega')\le R$.
\end{lemma}
\begin{proof}
Every surviving pair remains non-overlapping. If $g$ is the directed gap
between successive surviving centers, then
$\min\{g,2\pi-g\}\ge\theta_R(u,v)$, hence $g\ge\theta_R(u,v)$.
All these gaps, including closure, sum to $2\pi$. Thus the chain sum
at $R$ is at most $2\pi$; strict decrease gives the assertion.
This proof uses full pair feasibility before deletion. It never reconstructs
a full placement from a chain root.
\end{proof}

\section{Exact terminal parameters (L2)}\label{sec:terminal}
\begin{lemma}\label{lem:terminal}
The parameters in \eqref{eq:parameters} are uniquely defined, with
$3/17<q<1/5<\beta_1$. Furthermore
\begin{equation}\label{eq:J}
J:=\int_q^1\sqrt{x(1+q-x)}\dd x=\pi\Cterm.
\end{equation}
Among $0<r<1$ the expression
$\pi^{-1}\int_r^1\sqrt{x(1+r-x)}\dd x$ has a unique global maximum
at $r=q$.
\end{lemma}
\begin{proof}
The function $\cos t-t$ is continuous, positive at zero, negative at one
and strictly decreasing on $(0,1)$, since its derivative is $-1-\sin t$.
This proves existence and uniqueness of $\tau$. The rational gates in
Section~\ref{sec:arithmetic} prove $3/17<q<1/5$ without assuming rounded
approximations.

For arbitrary $r\in(0,1)$ put $v=(1-r)/(1+r)$. Substitution
$y=(2x-1-r)/(1+r)$ gives
\[
\frac1\pi\int_r^1\sqrt{x(1+r-x)}\dd x
=\frac{\arcsin v+v\sqrt{1-v^2}}{\pi(1+v)^2}=:\Phi(v).
\]
Here $r\mapsto v$ is a bijection of $(0,1)$ onto itself, reversing order.
Writing $A(v)=\arcsin v+v\sqrt{1-v^2}$ gives
\[
A'(v)=2\sqrt{1-v^2},\qquad
\Phi'(v)=\frac{2(\sqrt{1-v^2}-\arcsin v)}{\pi(1+v)^3}.
\]
With $t=\arcsin v$, the numerator changes sign exactly at $t=\tau$.
Thus $v=\sin\tau$ is the unique global maximum. The endpoint limits
are $0$ and $1/8$, and strict monotonicity on either side excludes
boundary maximizers. At this point $\sqrt{1-v^2}=\tau$, so
$A(v)=\tau(1+v)$ and $\Phi(v)=\tau/(\pi(1+v))$.
Since $1+q=2/(1+v)$, this proves \eqref{eq:J} and the claimed value.
Each $D_i$ is now an ordinary integral of a continuous, exactly specified
function over a uniquely specified interval. Its sign and enclosure are
proved arithmetically in Section~\ref{sec:arithmetic}.
\end{proof}
The optimization just made concerns a terminal integral. The lower-bound
proof below uses its identity \eqref{eq:J}; no theorem about an exactly
minimizing finite Supnick cycle is required as an additional premise.

\section{One outer order, its dual and its energy (L3)}\label{sec:dual}
Let $n\ge102$, $k=\lfloor qn\rfloor$, $\alpha=k/n$, $s=1+\alpha$,
$V=\{k/n,(k+1)/n,\ldots,1\}$. The strict bound $q>3/17$ gives
\begin{equation}\label{eq:gridgate}
\alpha>q-1/n>3/17-1/102=1/6,\qquad \alpha\le q<1/5.
\end{equation}
Fix one genuine cyclic order $\sigma$ of the original labels $k,\ldots,n$.
Use normalized coordinates $x=j/n$ in its edges and set
\[
\begin{gathered}
f(x,y)=\sqrt{xy},\qquad r(x)=s-x,\\
W_0=\frac1n\sum_{\{x,y\}\in\sigma} f(x,y),\quad
J_n=\frac1n\sum_{x\in V} f(x,r(x)),\quad e=W_0-J_n.
\end{gathered}
\]
Thus $W_0=S(\sigma)/n^2$. The reflection $r$ permutes $V$ exactly.

\begin{lemma}[Assignment dual]\label{lem:dual}
For $c=s/2$ define
\[
\psi(x)=\frac c2+\frac12\int_c^x\sqrt{\frac{s-t}{t}}\dd t.
\]
Then for $x,y\in[\alpha,1]$,
\begin{equation}\label{eq:dual}
\frac{(x+y-s)^2}{8}\le f(x,y)-\psi(x)-\psi(y)
\le\frac{(x+y-s)^2}{8\alpha}.
\end{equation}
Consequently
\begin{equation}\label{eq:E}
e\ge0,\qquad E:=\frac1n\sum_{\{x,y\}\in\sigma}(x+y-s)^2\le8e.
\end{equation}
\end{lemma}
\begin{proof}
Differentiation shows that $\psi(x)+\psi(s-x)$ and $f(x,s-x)$ have
equal derivatives; at $x=c$ both equal $c$. For
$H(x,y)=f(x,y)-\psi(x)-\psi(y)$ this gives $H(x,s-x)=0$, and
\[
\partial_yH(x,y)=
\frac{x+y-s}{2\sqrt y(\sqrt x+\sqrt{s-y})}.
\]
The denominator lies between $4\alpha$ and $4$. Integrate from $s-x$
to $y$. If $y<s-x$, the integration direction and numerator both
reverse sign. In both cases integrating the linear numerator gives
\eqref{eq:dual}. Reflection and degree two imply
$\sum_{x\in V}f(x,r(x))=2\sum_{x\in V}\psi(x)
=\sum_{\{x,y\}\in\sigma}(\psi(x)+\psi(y))$.
Summing \eqref{eq:dual} proves \eqref{eq:E}.
\end{proof}
The reference $J_n$ need not be the cost of a realizable cycle:
reflection can have two-cycles and a middle fixed point. Only its exact
assignment identity was used. No loop or duplicate neighbor is inserted
into $\sigma$.

Replace each cyclic edge by its two orientations and define
\begin{equation}\label{eq:measure}
\mu=\frac1{2n}\sum_{(x,y)\text{ oriented edge}}\delta_{(x,r(y))},
\qquad \nu=\frac1n\sum_{x\in V}\delta_x.
\end{equation}
Both marginals of $\mu$ are $\nu$, since each vertex occurs twice in
each endpoint position and reflection permutes the grid. Hence
\begin{equation}\label{eq:marginals}
\int[A(x)-A(z)]\dd\mu(x,z)=0,\qquad
\int(x-z)^2\dd\mu(x,z)=E
\end{equation}
for any function $A$ on $V$. The total mass is $(n-k+1)/n$, not
artificially renormalized to one. Every restriction below uses this
same $\sigma$, $\mu$ and $E$.

\section{Signed deletion, seams and finite errors (L4)}\label{sec:deletion}
Consider any $\beta\in(q,23/100]$ and let
$\ell=\lfloor\beta n\rfloor$. Assume in addition
\begin{equation}\label{eq:cutgate}
\ell/n>q,\qquad \ell\le n-2.
\end{equation}
These gates ensure a nonempty deleted set and at least three survivors.
Write $\gamma=\ell/n$ and $B=(\ell-1/2)/n$. A grid vertex is deleted
exactly when $x<B$. Since
$s-2\gamma>7/6-46/100=53/75>1/2$, deleted vertices lie below $s/2$.
Put
\begin{align*}
D_\alpha(x)&=s-x-2\sqrt{x(s-x)},&
D_n&=\frac1n\sum_{j=k}^{\ell-1}D_\alpha(j/n),\\
W_\beta&=\frac1n\sum_{\{x,y\}\in\sigma|\{\ell,\ldots,n\}}f(x,y),&
\Delta&=W_\beta-W_0-D_n,\\
K&=\int |x-z|\ind_{(x-B)(z-B)<0}\dd\mu(x,z).
\end{align*}

\begin{lemma}[Midpoint strip mass]\label{lem:strip}
For every $h>0$, $\nu\{x:|x-B|\le h\}\le4h$.
\end{lemma}
\begin{proof}
No grid point is within $1/(2n)$ of $B$. Below this radius the set is
empty. Otherwise it has at most $2nh+1$ points, of total mass at most
$2h+1/n\le4h$. Truncating the grid at either endpoint only reduces mass.
\end{proof}

\begin{lemma}[Signed first variation]\label{lem:variation}
Let
\[
p(x)=\tfrac12\left(1-\sqrt{x/(s-x)}\right),\quad
P(x)=p(x)\ind_{x<B},\quad
\Lambda=\frac1n\sum_{x<B}p(x)\sum_{y\sim x}(x+y-s).
\]
The neighbors $y\sim x$ are those in the original outer cycle. Then
\begin{equation}\label{eq:43}
\Lambda=2\int P(x)(x-z)\dd\mu,\qquad
|\Delta-\Lambda|\le43E.
\end{equation}
\end{lemma}
\begin{proof}
The first identity follows directly from the oriented mass $1/(2n)$.
Partition deleted vertices into maximal cyclic runs. Rotate the written
cycle to a survivor if necessary. Every removed edge and every replacement
edge is then accounted for exactly once, including the seam.

For an isolated deleted vertex $x$ with surviving neighbors $y,z$ the
change in unnormalized $f$-cost is
$g(x;y,z)=f(y,z)-f(x,y)-f(x,z)$. At $y=z=t=s-x$ it equals
$D_\alpha(x)$, and both first derivatives in $y,z$ equal $p(x)$.
The Hessian entries are
\[
g_{yy}=\frac{\sqrt{x}-\sqrt z}{4y^{3/2}},\quad
g_{zz}=\frac{\sqrt{x}-\sqrt y}{4z^{3/2}},\quad
g_{yz}=\frac1{4\sqrt{yz}}.
\]
On $[\alpha,1]^3$ their absolute values are respectively less than
$4,4,2$: $1/(4\alpha^{3/2})<3\sqrt6/2<4$ and
$1/(4\alpha)<3/2<2$. The symmetric Hessian therefore has operator
norm at most $6$ (each absolute row sum is at most $6$). Taylor's
integral remainder along the segment from $(t,t)$ to $(y,z)$ yields
\[
|g-D_\alpha(x)-p(x)(y-t+z-t)|
\le3((y-t)^2+(z-t)^2).
\]
The segment stays in $[\alpha,1]^2$. Across isolated deletions each
incident edge occurs at most once, so their total normalized error is
at most $3E$.

For runs of length $m\ge2$, let $M_{LL}$ be the total number of outer
edges with both endpoints deleted and $V_{\rm bad}$ the number of
vertices in these runs. Since $m\le2(m-1)$,
$V_{\rm bad}\le2M_{LL}$. A run replaces $m+1$ edges by one edge;
all weights lie in $[0,1]$, so its cost change has absolute value at
most $m+1\le2m$. Since $|D_\alpha(x)|\le2$, its discrepancy from
the $m$ ideal changes is at most $4m$. The total normalized discrepancy
of these runs is at most $8M_{LL}/n$.
For a deleted vertex, $0<p(x)<1/2$ and each defect has absolute value
at most $1-\alpha<1$. The part of $\Lambda$ from bad vertices therefore
has absolute value at most $V_{\rm bad}/n\le2M_{LL}/n$; this includes
both incidences of each $LL$ edge. On each such edge
$|x+y-s|>1/2$, so $M_{LL}/n\le4E$. Combining contributions gives
$3E+(8+2)M_{LL}/n\le43E$.
\end{proof}

\begin{lemma}[Cancellation and signed deletion]\label{lem:54}
Under the same hypotheses,
\begin{equation}\label{eq:54}
|\Lambda|\le11E+2K,\qquad
|W_\beta-W_0-D_n|\le54E+2K.
\end{equation}
\end{lemma}
\begin{proof}
Throughout $[\alpha,1]$, $|p(x)|<1$ and
\[
p'(x)=-\frac{s}{4\sqrt x(s-x)^{3/2}},\qquad
|p'(x)|\le\frac{s}{4\alpha^2}<\frac{54}{5}<11.
\]
The first bound follows from $\sqrt{x/(s-x)}<\sqrt6<3$;
the second uses $s<6/5$ and $\alpha>1/6$.
Let $A(u)=\int_\alpha^u P(v)\dd v$, a continuous primitive with a
piecewise continuous derivative. For grid points $x,z$, integrate
\[
P(x)-P(v)=\ind_{x<B}(p(x)-p(v))
+p(v)(\ind_{x<B}-\ind_{v<B})
\]
along the segment from $z$ to $x$. The smooth term contributes at most
$11|x-z|^2/2$. The jump term vanishes unless the segment crosses $B$
and otherwise contributes at most $|x-z|$. Both orientations give
\[
|P(x)(x-z)-(A(x)-A(z))|
\le\tfrac{11}{2}(x-z)^2+|x-z|\ind_{(x-B)(z-B)<0}.
\]
Neither endpoint equals the midpoint $B$. Equal marginals in
\eqref{eq:marginals} cancel the primitive exactly. Multiply the remaining
bound by two and use \eqref{eq:43} to obtain $11E+2K$, then add $43E$.
Thus the constant is $43+11=54$, with no mesh error in this estimate.
\end{proof}

\begin{lemma}[Uniform floor estimates]\label{lem:floors}
Set $D(\beta)=\int_q^\beta D_q(x)\dd x$.
For every cutoff satisfying \eqref{eq:cutgate},
\begin{equation}\label{eq:floorerrors}
|J_n-\pi\Cterm|\le5/n,\qquad |D_n-D(\beta)|\le19/(2n).
\end{equation}
\end{lemma}
\begin{proof}
Let $F_\alpha(x)=\sqrt{x(1+\alpha-x)}$. All factors in the
comparisons below lie in $[1/6,1]$, and direct differentiation gives
\[
|\partial_xF_\alpha|\le3,\quad
|\partial_\alpha F_\alpha|\le3/2,\quad |F_\alpha|\le1,
\qquad
|\partial_xD_\alpha|\le7,\quad
|\partial_\alpha D_\alpha|\le2,\quad |D_\alpha|\le2.
\]
For example each coordinate derivative of $\sqrt{xt}$ is at most
$\sqrt6/2<3/2$, and
$\partial_\alpha D_\alpha=1-\sqrt{x/(1+\alpha-x)}$.
For a Lipschitz function with constant $M$, the left-rectangle error
on each mesh interval is at most $M/(2n^2)$. There are at most $n$
intervals. The inclusive sum $J_n$ also has its last endpoint term;
its error against $\int_\alpha^1F_\alpha$ is at most $(3/2+1)/n$.
Removing $[\alpha,q]$ costs at most $1/n$ and changing the parameter
from $\alpha$ to $q$ on $[q,1]$ costs at most $3/(2n)$. Their sum is $5/n$.

The sum $D_n$ is exactly a left sum on $[\alpha,\gamma]$; its error
is at most $7/(2n)$. Removing $[\alpha,q]$, changing the parameter on
$[q,\gamma]$ and adding $[\gamma,\beta]$ cost at most $2/n$ each.
Here $\gamma>q$ is an explicit hypothesis. Their sum is $19/(2n)$.
Only $0\le q-\alpha<1/n$ and $0\le\beta-\gamma<1/n$ are used.
No numerical evaluation of a transcendental floor or parity selection occurs.
\end{proof}

\section{One shared crossing budget (L5)}\label{sec:crossings}
\begin{lemma}\label{lem:crossing}
Let $B_1<\cdots<B_m$, $h_i>0$, and let $\mu$ be a finite nonnegative
measure with first marginal $\nu$, finite energy
$E=\int|x-z|^2\dd\mu$ and strip bounds
$\nu\{|x-B_i|\le h_i\}\le4h_i$. Define
$K_i=\int|x-z|\ind_{(x-B_i)(z-B_i)<0}\dd\mu$.
If $h_i+h_{i+1}\le B_{i+1}-B_i$ for every adjacent pair, then
\begin{equation}\label{eq:shared}
\sum_{i=1}^m h_iK_i\le4\sum_{i=1}^m h_i^3+E.
\end{equation}
\end{lemma}
\begin{proof}
Write $d=|x-z|$. At cutoff $i$ split crossings into short ($d\le h_i$)
and long ($d>h_i$). A short crossing has $|x-B_i|\le h_i$, so its
weighted integrand is at most $h_i^2\ind_{|x-B_i|\le h_i}$. Integrating
and summing gives at most $4\sum h_i^3$.
The strictly crossed cutoffs form a consecutive block $p,\ldots,r$.
With no crossing the long contribution is zero; with one it is at most
$h_pd<d^2$ if long. If $r>p$, summing adjacent separation gives
\[
\sum_{i=p}^r h_i
=\sum_{i=p}^{r-1}(h_i+h_{i+1})-\sum_{i=p+1}^{r-1}h_i
\le B_r-B_p-\sum_{i=p+1}^{r-1}h_i\le B_r-B_p<d.
\]
Thus every crossing in this block is long, and their combined weighted
integrand is $d\sum_{i=p}^rh_i\le d^2$. Integration spends $E$ once,
even when the same pair crosses all three thresholds. The equality case
$d=h_i$ was included in the short part; endpoints at a cutoff do not
strictly cross it. No pair is omitted or charged to several energy budgets.
\end{proof}

\section{Finite minimax and transfer to the global problem (L6)}\label{sec:minimax}
Use the fixed three cutoffs of \eqref{eq:parameters}. Let
$\ell_i=\lfloor\beta_i n\rfloor$, $B_i=(\ell_i-1/2)/n$ and
$\sigma_i=\sigma|\{\ell_i,\ldots,n\}$, with $\sigma_0=\sigma$.
Put $W_i=S(\sigma_i)/n^2$ and let $D_{i,n}$ denote $D_n$ at cutoff $i$.
Define the common-order minimax
\[
M_{n,4}=\min_{\sigma\text{ cyclic on }\{\lfloor qn\rfloor,\ldots,n\}}
\max_{0\le i\le3}\Rc(\sigma_i).
\]
Only sufficiently large $n$ are needed; all its cycles then have at least
three positive radii. For positive widths define
\[
H=\sum_i h_i,\quad Q_3=\sum_i h_i^3,\quad L=16+432H,
\quad F_n=\sum_i h_iD_{i,n}-8Q_3.
\]

\begin{proposition}\label{prop:finite}
Suppose
\begin{equation}\label{eq:finitegates}
n\ge102,\quad \ell_1/n>q,\quad
\ell_1<\ell_2<\ell_3\le n-2,\quad
h_i+h_{i+1}\le\frac{\ell_{i+1}-\ell_i}{n}\quad(i=1,2).
\end{equation}
Then
\begin{align}
\Rstar(n)&\ge M_{n,4},\label{eq:globalfinite}\\
\frac{M_{n,4}}{n^2}&\ge\frac{J_n+(F_n)_+/L}{\pi}-\frac1n
\ge\Cterm+\eta(h)-\frac{A(h)}n,\label{eq:finite}\\
A(h)&=1+\frac{5+19H/(2L)}\pi.\nonumber
\end{align}
\end{proposition}
\begin{proof}
All three cutoffs satisfy Section~\ref{sec:deletion}'s hypotheses. With
the same outer measure, Lemmas~\ref{lem:54} and \ref{lem:crossing} give
\begin{align*}
\sum_i h_iW_i
&\ge H(J_n+e)+\sum_i h_iD_{i,n}-54HE-2\sum_i h_iK_i\\
&\ge H(J_n+e)+F_n-(54H+2)E\\
&\ge H(J_n+e)+F_n-Le.
\end{align*}
Here $L=8(54H+2)$ because $E\le8e$; the factor $16$ is not
multiplied by the number of cutoffs. Retaining the outer cost and the
positive weighted mean of the inner costs yields
\[
\max_{0\le i\le3}W_i\ge J_n+
\max\left\{e,\frac{F_n}{H}-\left(\frac LH-1\right)e\right\}.
\]
Since $L>H>0$, if $F_n>0$ the increasing and decreasing branches meet
at $e=F_n/L\ge0$. Below that point the decreasing branch is at least
its meeting value, and above it the increasing branch is at least that
value. If $F_n\le0$, the first branch is always nonnegative and $e=0$
makes the maximum zero. The infimum over all $e\ge0$ is therefore
$(F_n)_+/L$ in every case. This scalar infimum need not be realized by
a tour. The lower sandwich \eqref{eq:sandwich} and a finite minimum
over orders prove the first inequality in \eqref{eq:finite}.

Lemma~\ref{lem:floors} gives $F_n\ge F(h)-19H/(2n)$ and
$J_n\ge\pi\Cterm-5/n$. The positive-part function is increasing and
1-Lipschitz, so $(F_n)_+\ge F(h)_+-19H/(2n)$. This proves the second
inequality even when $F_n$ or $F(h)$ is nonpositive.

For \eqref{eq:globalfinite}, take any full feasible original placement
at radius $R$. Delete the labels below $\lfloor qn\rfloor$ to obtain
one common order $\sigma$. Its three further restrictions are exactly
$\sigma_i$. Lemma~\ref{lem:deletiongeometry} bounds each chain root by
$R$. Thus $M_{n,4}\le R$ for every full feasible placement. Taking the
infimum over $R$ proves the assertion without an attainment assumption.
\end{proof}
All estimates are uniform in the order, include the seam and consecutive
deleted runs, and apply to both cardinality parities. No implication from
chain feasibility to full feasibility was used.

\section{The unique global width optimum (L7)}\label{sec:width}
Set $d=a+b$ and $\ell=16+432d$; these are width constants, not grid
parameters. On the edge where both separation constraints are active set
\begin{align}
N(z)&=F(a-z,z,b-z)=c_0+c_1z-24dz^2+8z^3,\nonumber\\
c_0&=aD_1+bD_3-8(a^3+b^3),\qquad
c_1=D_2-D_1-D_3+24(a^2+b^2),\nonumber\\
L(z)&=\ell-432z,\nonumber\\
G(z)&=N'(z)L(z)+432N(z)\label{eq:G}\\
&=\ell c_1+432c_0-48\ell dz+(24\ell+10368d)z^2-6912z^3.\nonumber
\end{align}
Its exact derivative is
\begin{equation}\label{eq:Gprime}
G'(z)=N''(z)L(z)=48(z-d)L(z)<0\qquad(0\le z\le a),
\end{equation}
since $a<d$ and $L(z)\ge16+432b>0$.

The rational computations proved in Section~\ref{sec:arithmetic} give
$G(0)>0$, $G(a)<0$ and the more precise opposite signs at
\begin{equation}\label{eq:zbracket}
z_-=0.00451910758124826,\qquad z_+=0.00451910758124827.
\end{equation}
Together with \eqref{eq:Gprime}, continuity proves the existence and
uniqueness of $z_*\in(z_-,z_+)$. All coordinates of
$u=(a-z_*,z_*,b-z_*)$ are strictly positive.

\begin{lemma}[Global cubic certificate]\label{lem:KKT}
Let $\rho=F(u)/\mathcal L(u)$ and
$g_i=D_i-24u_i^2-432\rho$. Then $F(u)>0$ and
\[
\lambda_1=g_1>0.00035,\qquad \lambda_2=g_3>0.00059,
\qquad(g_1,g_2,g_3)=(\lambda_1,\lambda_1+\lambda_2,\lambda_2).
\]
For every $h\in\mathcal P$,
\begin{align}\label{eq:cubiccertificate}
F(h)-\rho\mathcal L(h)
={}&\lambda_1(h_1+h_2-a)+\lambda_2(h_2+h_3-b)\nonumber\\
&-8\sum_{i=1}^3(h_i-u_i)^2(h_i+2u_i)\le0,
\end{align}
with strict inequality unless $h=u$. Thus $u$ is the unique global
maximizer of $\eta$ on the entire closed polyhedron.
\end{lemma}
\begin{proof}
The positivity gates are rational consequences of the integral and root
intervals in Section~\ref{sec:arithmetic}. Since $G(z_*)=0$, division
by $L(z_*)$ gives $N'(z_*)+432\rho=0$. On substituting $N'$ this is
$-g_1+g_2-g_3=0$. The gradient of $F-\rho\mathcal L$ at $u$ is
therefore the stated positive combination of the two constraint normals.
For each coordinate the exact identity
\[
h_i^3-u_i^3-3u_i^2(h_i-u_i)=(h_i-u_i)^2(h_i+2u_i)
\]
gives \eqref{eq:cubiccertificate}, since $F(u)-\rho\mathcal L(u)=0$
and the two constraints are equalities at $u$. The linear terms are
nonpositive on $\mathcal P$. Because $h_i\ge0$ and $u_i>0$, the cubic
remainder is strictly negative if any coordinate differs. Hence
$F(h)/\mathcal L(h)<\rho$ for $h\ne u$. If $F(h)\le0$, the
positive-part objective is zero and is also strictly smaller than
$\rho/\pi>0$. This covers zero widths, all faces and vertices, and the
whole nonpositive region without assuming beforehand that an optimum
must lie on the chosen edge. Existence also follows directly from
compactness of $\mathcal P$ and continuity of $\eta$.
\end{proof}

\section{Reproducible rational arithmetic (L2 and L7)}\label{sec:arithmetic}
The included \texttt{check\_width\_optimum.py} uses only Python's standard
library integers and \texttt{Fraction}. It starts from the rational
intervals below, proves that they contain the named exact constants,
and does not import a saved numerical result or an optimizer. To make
the arithmetic reproducible by an independent implementation, we specify
every enclosure formula.

\subsection{Trigonometric gates}
Use the intervals
\begin{align*}
T_-&=0.7390851332151606416553120876738734040134,\\
T_+&=0.7390851332151606416553120876738734040135,\\
q_-&=0.1950200913506069300798071259134151019366,\\
q_+&=0.1950200913506069300798071259134151019367,\\
\pi_-&=3.1415926535897932384626433832795028841971,\\
\pi_+&=3.1415926535897932384626433832795028841972.
\end{align*}
Let
\[
C_m(x)=\sum_{j=0}^{m-1}\frac{(-1)^jx^{2j}}{(2j)!},\quad
S_m(x)=\sum_{j=0}^{m-1}\frac{(-1)^jx^{2j+1}}{(2j+1)!},\quad
A_m(x)=\sum_{j=0}^{m-1}\frac{(-1)^jx^{2j+1}}{2j+1}.
\]
For $0<x<1$ the terms decrease to zero, so an even number of terms
is a lower bound and an odd number is an upper bound. Exact rational
comparisons give
\[
C_{30}(T_-)>T_-,\qquad C_{31}(T_+)<T_+.
\]
Monotonicity of $\cos t-t$ isolates $\tau$. Sine is increasing on
$(0,1)$, so $s_-=S_{30}(T_-)<\sin\tau<S_{31}(T_+)=s_+$.
The decreasing map $s\mapsto(1-s)/(1+s)$ and exact comparisons give
\[
q_-<\frac{1-s_+}{1+s_+}<q<\frac{1-s_-}{1+s_-}<q_+.
\]
In particular $q_->3/17$ and $q_+<1/5$ are checked as rational
inequalities.

For $\pi$ use Machin's identity
$\pi=16\arctan(1/5)-4\arctan(1/239)$. Here is a branch justification:
two tangent doublings give $\tan(4\arctan(1/5))=120/119$,
and the tangent of $4\arctan(1/5)-\arctan(1/239)$ is exactly one.
This angle lies strictly between zero and $4/5<\pi/2$ (the elementary
inequality $\pi>2$ follows, for example, by comparing a semicircle
with its diameter). Positivity follows from
$\arctan(1/5)>1/5-(1/5)^3/3$ and $\arctan(1/239)<1/239$.
The angle is therefore $\pi/4$. Alternating sums give the exact gates
\[
\pi_-<16A_{32}(1/5)-4A_{31}(1/239)
<\pi<16A_{31}(1/5)-4A_{32}(1/239)<\pi_+.
\]

\subsection{Integral certificates}
For each $\beta_i$ use $s_0=1+q_-$,
$v=(1-q_-)/s_0$, $w=(s_0-2\beta_i)/s_0$. These satisfy $0<w<v<1$.
The convergent binomial expansion is
\[
\sqrt{1-y^2}=1-\sum_{j\ge1}c_jy^{2j},\qquad
c_j=\frac{\binom{2j}{j}}{4^j(2j-1)}>0.
\]
Its coefficients decrease, since $c_{j+1}/c_j=(2j-1)/(2j+2)<1$.
Termwise integration is valid uniformly on $[w,v]$, which lies inside
the radius of convergence. Define the rational quantities
\begin{align*}
P&=v-w-\sum_{j=1}^{80}c_j\frac{v^{2j+1}-w^{2j+1}}{2j+1},\\
T&=(v-w)c_{81}\frac{v^{162}}{1-v^2},\qquad
I_-=\frac{s_0^2}{4}(P-T),\quad I_+=\frac{s_0^2}{4}P.
\end{align*}
For every $y\in[w,v]$ the omitted positive tail is at most
$c_{81}v^{162}/(1-v^2)$, hence
\[
I_-\le\int_{q_-}^{\beta_i}\sqrt{x(1+q_--x)}\dd x\le I_+.
\]
Write $D(r,\beta)=\int_r^\beta D_r(x)\dd x$. The endpoint and
integrand bounds are $|D_r(r)|\le2$ and $|\partial_rD_r(x)|\le2$.
Differentiation under the integral therefore gives
$|\partial_rD(r,\beta)|\le4$ on $r\in[q_-,q_+]$.
Thus valid rational endpoints for $D_i$ are
\begin{align*}
d_{i,-}&=s_0(\beta_i-q_-)-(\beta_i^2-q_-^2)/2-2I_+-4(q_+-q_-),\\
d_{i,+}&=s_0(\beta_i-q_-)-(\beta_i^2-q_-^2)/2-2I_-+4(q_+-q_-).
\end{align*}
Before further propagation the script enlarges these endpoints to
30-place rationals: for an interval $[l,r]$ it uses
\[
\left[\frac{\lfloor10^{30}l\rfloor-1}{10^{30}},
\frac{\lceil10^{30}r\rceil+1}{10^{30}}\right].
\]
The extra unit makes the printed enclosures strict; it is an outward
rounding, never a decimal approximation used as an exact input value.

\subsection{Polynomial signs, multipliers and final enclosure}
Interval addition and subtraction use matching and opposite endpoints;
multiplication takes the minimum and maximum of all four products.
Division by a positive interval multiplies by its reversed reciprocal.
Positive powers use monotonic endpoints. These operations remain safe
when a variable appears repeatedly, although dependency can widen them.
Evaluate $G$ at $0,a,z_-,z_+$ using the intervals for $D_i$.
Their signs are respectively $+,-,+,-$ with strict rational margins.
With $Z=[z_-,z_+]$ use the interval vector
\[
U=([a-z_+,a-z_-],\ Z,\ [b-z_+,b-z_-]).
\]
Evaluate $F(U)$, $\rho=F(U)/\mathcal L(U)$ and
$D_i-24U_i^2-432\rho$ outwardly. These certify $F(u)>0$ and the
two positive multipliers in Lemma~\ref{lem:KKT}. The accompanying
\texttt{certificate.json} gives the rational endpoints of these gates,
including the four signs, so they can be checked without interpreting
floating output. The enclosed quantities include
\begin{align*}
0.00136820131190299900276&<D_1<0.00136820131190299900279,\\
0.00196196296152142894680&<D_2<0.00196196296152142894683,\\
0.00241410289623904895464&<D_3<0.00241410289623904895467,\\
0.00000038803240421408705&<\etaw<0.00000038803240421408841.
\end{align*}
Finally enclose $\Cterm=\tau(1+q)/(2\pi)$ by multiplying the positive
intervals for $\tau,1+q$ and dividing by opposite endpoints of $2\pi$.
Add the interval for $\rho/\pi$. Exact comparisons with the two rationals
in \eqref{eq:mainbound} prove the strict enclosure in that equation.
The script \texttt{verify.py} checks those literal target endpoints, the
certificate content and the input/output manifest; a replay that merely
prints a plausible decimal is not the acceptance condition.

\section{The ordered passage to the limit (L8)}\label{sec:limit}
For any fixed positive width vector with strict macroscopic margins
$\delta_i=\beta_{i+1}-\beta_i-h_i-h_{i+1}>0$, choose an integer $N$
such that $N\ge102$, $N(\beta_1-q)>1$, and $N\delta_i\ge1$.
For every $n\ge N$,
\[
\ell_1/n>\beta_1-1/n>q,\qquad
\frac{\ell_{i+1}-\ell_i}{n}>
\beta_{i+1}-\beta_i-1/n\ge h_i+h_{i+1}.
\]
Also $n-\ell_3\ge77n/100\ge2$. The gates \eqref{eq:finitegates}
hold for every such integer $n$. Proposition~\ref{prop:finite} therefore
implies, at fixed $h$,
\begin{equation}\label{eq:fixedlimit}
\liminf_{n\to\infty}\frac{M_{n,4}}{n^2}\ge\Cterm+\eta(h),\qquad
\liminf_{n\to\infty}\frac{\Rstar(n)}{n^2}\ge\Cterm+\eta(h).
\end{equation}

The optimum $u$ saturates both macroscopic constraints and cannot be
substituted blindly into the finite gates. For any fixed $0<t<1$, use
$h=tu$. Its strict margins are $(1-t)a$ and $(1-t)b$. The rational
gates $\beta_1-q>a$ and $1/a>102$ show that
\[
N(t)=\left\lceil\frac1{(1-t)a}\right\rceil
\]
is a sufficient threshold. Apply \eqref{eq:fixedlimit} with this fixed
$t$ and only then let $t\uparrow1$. By continuity $\eta(tu)\to\eta(u)$,
so
\begin{equation}\label{eq:liminf}
\liminf_n M_{n,4}/n^2\ge\Cterm+\etaw,\qquad
\liminf_n\Rstar(n)/n^2\ge\Cterm+\etaw.
\end{equation}
The quantity $M_{n,4}$ depends on the three cutoffs and common order,
not on the auxiliary widths. The same liminf is bounded for every
fixed $t$; no minimum over orders has been exchanged with a limit.

For completeness, unchanged $u$ fails both finite separation gates at
every $n=50000j+1$, $j\ge1$. The exact floors are
\[
(\ell_1,\ell_2,\ell_3)=(10455j,10907j,11500j),
\]
and their normalized differences are respectively $a(n-1)/n<a$ and
$b(n-1)/n<b$. This explains why the order of limits is essential.
It is not a failure of \eqref{eq:liminf}.

\section{Existence of the global limit}\label{sec:existence}
We include only the line comparison and concatenation argument needed
to identify the liminf in \eqref{eq:liminf} with $C_*$. No finite-word
linear programming characterization or upper-bound construction is needed.

For a finite list of nonnegative marks $v_1,\ldots,v_N$, let $b(v)$ be
the least span of line points carrying those marks, with
\begin{equation}\label{eq:line}
|x_i-x_j|\ge\sqrt{v_iv_j}\quad(i\ne j).
\end{equation}
Equal marks still label distinct objects; zero marks may coincide.
For a fixed order $w$, the minimal increasing placement is
\[
x_1=0,\qquad x_i=\max_{j<i}\{x_j+\sqrt{w_iw_j}\}.
\]
It satisfies all constraints. Induction shows that every feasible
placement in this order starting at zero has coordinates at least these.
Thus the minimum span for a fixed order is attained; minimizing over
finitely many orders proves attainment of $b(v)$. Deleting or decreasing
marks cannot increase $b$; scaling marks by $h\ge0$ scales $b$ by $h$.
For $h>0$ the latter follows by scaling positions in both directions,
and for $h=0$ both sides are zero. These assertions also hold for marks
temporarily larger than one.

For marks in $[0,1]$, let $B(v)$ be the infimum circumference of a circle
on which both directed arcs between each pair are at least
$\sqrt{v_iv_j}$. Cutting a feasible circle gives a feasible line
placement. Conversely close an optimal line placement with a gap of
length one: one arc is the line distance and the complementary arc
contains the new gap, at least $1\ge\sqrt{v_iv_j}$. Hence
\begin{equation}\label{eq:bB}
b(v)\le B(v)\le b(v)+1.
\end{equation}
Write $b_n=b(1/n,2/n,\ldots,1)$. From a full feasible original circle
at radius $R$, scale angles by $(R+n)/(2n)$. The lower kernel bound
in Lemma~\ref{lem:sandwich} produces a $B$-circle of circumference
$\pi(R+n)/n$. Conversely a $B$-circle of circumference $b_n+1$ exists
by the preceding gap construction. Scale it to angular circumference
$2\pi$ and set $R=n(b_n+1)/\pi$. Every directed arc is then at least
$2\sqrt{uv}/R$, hence at least $\theta_R(u,v)$ by the upper kernel
bound. The cosine law gives full original feasibility for all pairs.
Taking infima proves for every $n\ge3$
\begin{equation}\label{eq:globalsqueeze}
\frac{b_n}{\pi n}-\frac1n\le\frac{\Rstar(n)}{n^2}
\le\frac{b_n+1}{\pi n}.
\end{equation}

\begin{lemma}\label{lem:existence}
The limit $E_\infty=\lim_n b_n/n$ exists and equals
$\inf_{n\ge2}(b_n+1)/n$. Consequently $C_*=E_\infty/\pi$ exists.
\end{lemma}
\begin{proof}
Fix $n\ge2$. For any $N\ge n$, put $k=\lceil N/n\rceil$ and
concatenate $k$ copies of an optimal $b_n$ arrangement with inter-block
gaps one. All cross-block distances are at least one, so all pair
constraints hold. Its span is $k(b_n+1)-1$, with exactly $k$ copies
of each of the $n$ marks. Retain the $N$ smallest marks. In sorted
mark order their $i$th value is $\lceil i/k\rceil/n$.
Scale all positions and marks by $h=nk/N$, so this value is at least
$i/N$. Decrease the retained marks to $i/N$, assigning each label once
(ties may be broken arbitrarily). Feasibility is preserved. Thus
\[
b_N\le\frac{nk}{N}\,[k(b_n+1)-1].
\]
Since $k/N\to1/n$, division by $N$ gives
$\limsup_N b_N/N\le(b_n+1)/n$ for every fixed $n$.
Equally spaced points with gaps one give $0\le b_N/N\le1$.
Choose $n$ along an unbounded subsequence attaining the liminf; then
$\limsup_N b_N/N\le\liminf_n b_n/n$, proving existence.
The preceding upper bounds and $n\to\infty$ identify the limit with
the displayed infimum. Equation~\eqref{eq:globalsqueeze} now proves
the final assertion for the full geometric problem.
\end{proof}
Combining Lemma~\ref{lem:existence}, \eqref{eq:liminf} and the rational
enclosure proves Theorem~\ref{thm:main}. The concatenation construction
uses actual labeled points and every integer $N$, not only a subsequence
or a coupling with correct marginals.

\section{Verification scope and provenance}\label{sec:scope}
The proof establishes exact theorems and corollaries. The parameter
enclosures are finite rational computations supporting their stated
signs and constants. Bounded tour and measure checks test implementation
invariants and failure modes; they do not prove the universal statements
by enumeration. Neither endpoint of a geometric asymptotic interval is
claimed sharp. The optimized width value is sharp only for the function
$\eta$ on the specified polyhedron at the fixed cutoffs.

The analytic provenance is the immutable scientific snapshot
\begin{center}\small\texttt{3beb8d70c5b3748d370a92855847bdf574e5a14f}.\end{center}
The four source notes are identified in the bibliography. The published
target is\par\smallskip
{\small\path{paper_assets/asymptotic_sequel/ringmin_asymptotic.tex}},\par\smallskip
Proposition 7; the older finite-paper v2 source is not used as the target.
The package includes a dependency map, \path{DEPENDENCIES.md}, recording source sections,
hypotheses and verification limits. Its \texttt{MANIFEST.json} records
input and output SHA-256 hashes, and \texttt{VERIFICATION.md} reports
only checks actually run. These supporting files are included alongside
this source and PDF; no checkout history or old task dossier is required.


\maketitle
\linenumbers
\begin{abstract}
This supplement gives a complete proof of the reflected-block upper
construction and of the implicit parameters defining
$U_4=C_4(1/20000)$. It includes the all-pairs angular criterion, recovery
by actual integer labels, a quantitative radius estimate, uniqueness of
the successive parameters, and the rational gates used to isolate them.
The universal statements are analytic theorems; the accompanying bounded
rational computations establish specified constant inequalities. The
construction is an upper bound and does not assert an optimal partition
or a finite geometric optimum. These are existing results reorganized
from the source proofs listed in the dependency document.
\end{abstract}
\tableofcontents

\section{The Geometric Problem and the Target Bound}
Circles of radii $1,\ldots,n$ are externally tangent to a central circle
of radius $R>0$. Let $\Rstar(n)$ be the infimum of feasible central radii.
For a prescribed cyclic order $\sigma$, let $R_{\rm full}(\sigma)$ use
all pairs and let $R_{\rm chain}(\sigma)$ use only adjacent gaps. The
cosine law gives the required angle
\begin{equation}\label{eq:angle}
 \thR(a,b)=2\arcsin\sqrt{\frac{ab}{(R+a)(R+b)}}.
\end{equation}
Both directed arcs between each pair must be at least this angle. This is
equivalent to Cartesian non-overlap, since their sum is $2\pi$ and the
shorter angle lies in $[0,\pi]$. Throughout, inverse trigonometric
functions have their principal real branches. The kernel increases in
each outer radius and decreases strictly in $R$. In general
$R_{\rm chain}\le R_{\rm full}$; no equality is presumed here.

Define, on $\mathcal D=[0,1]\times[1,2]^2$,
\begin{equation}\label{eq:g}
 g(t,X,Y)=\max\{\sqrt t(\sqrt X+\sqrt Y),\sqrt{XY}\},\qquad
 C(\mu)=\frac1{4\pi}\int g\dd\mu.
\end{equation}
Sections~\ref{sec:x}--\ref{sec:widths} uniquely define $x_*,\al,
\epsb,\Ds$. Set $A=1+\al$, $\lambda=Ax_*$,
$v=\lambda+\epsb$, and $w=v+\Ds$. The measure $\mu_4(\eta)$ consists
of four separate symmetric reflections on the consecutive intervals of
lengths $(\lambda,\epsb,\Ds,\eta)$, followed by the shifted diagonal;
its precise general definition is \eqref{eq:mu}. Put
$C_3(d)=C(\mu_{\al;(\lambda,\epsb,d)})$ and
$C_4(\eta)=C(\mu_4(\eta))$.
The final result proved here is
\begin{equation}\label{eq:target}
 \limsup_{n\to\infty}\frac{\Rstar(n)}{n^2}
 \le U_4:=C_4(1/20000)
 <C_3(\Ds)-\frac1{4608000000000000}.
\end{equation}
If $C_*$ denotes the global limit established in the main article,
the left side is $C_*$. Its existence is not needed for this supplement's
upper-bound proof. Terminating decimal bounds below are exact rationals,
never definitions of the implicit parameters.

\section{An Exact All-Pairs Criterion}
\begin{lemma}[High-Shell Inequality]\label{lem:shell}
If $0<u\le v\le2u$, then $2\thR(u,u)>\thR(v,v)$.
Consequently every path $H_0,\ldots,H_j$ in $[u,v]$, $j\ge1$, satisfies
$\sum_{i=1}^j\thR(H_{i-1},H_i)\ge\thR(H_0,H_j)$.
\end{lemma}
\begin{proof}
Write $x=u/(R+u)$ and $z=v/(R+v)\le2x/(1+x)$. If $x\ge1/\sqrt2$,
then $2\arcsin x\ge\pi/2>\arcsin z$. Otherwise
\[
 (1+x)^2(1-x^2)-1=x(2-2x^2-x^3)>0,
\]
because $x^2<1/2$, $x<3/4$, and $2-2x^2-x^3>5/8$.
Thus $z<2x\sqrt{1-x^2}=\sin(2\arcsin x)$. Both comparison angles
are in $(0,\pi/2)$, so the same strict inequality follows. Multiplying
by two proves the first assertion. For any $a,b,c\in[u,v]$,
\[
 \thR(a,b)+\thR(b,c)\ge2\thR(u,u)>\thR(v,v)\ge\thR(a,c).
\]
Induction contracts a path. The one-edge case is equality.
\end{proof}

Let $m\ge2$ and let $(P_1,\ldots,P_m)$ be any permutation of
$\{m+1,\ldots,2m\}$. Subscripts are cyclic, with $P_0=P_m$.
The order is $(1,P_1,2,P_2,\ldots,m,P_m)$. Define
\begin{equation}\label{eq:cells}
 a_i=\thR(P_{i-1},i),\quad b_i=\thR(i,P_i),\quad
 c_i=\thR(P_{i-1},P_i),\quad d_i=\max\{a_i+b_i,c_i\}.
\end{equation}
\begin{theorem}[Full Cell Criterion]\label{thm:cells}
The order is feasible exactly when $S_P(R):=\sum_i d_i\le2\pi$.
More precisely its feasible gap vectors are exactly
\begin{equation}\label{eq:gapset}
 x_i\ge a_i,\quad y_i\ge b_i,\quad x_i+y_i\ge c_i,
 \qquad\sum_i(x_i+y_i)=2\pi,
\end{equation}
where $x_i$ is the gap $P_{i-1}\to i$ and $y_i$ is $i\to P_i$.
The full minimum is the unique positive root $\rho_P$ of $S_P(R)=2\pi$.
\end{theorem}
\begin{proof}
Necessity follows from the two adjacent constraints and the high-pair
constraint on each whole cell. For sufficiency first take any positive
gaps satisfying the three local inequalities, without requiring closure.
For each pair choose either simple cyclic path; the following arguments
apply separately in both directions.

For two highs, the path consists of whole cells. Replacing their lengths
by $c_i$ and applying Lemma~\ref{lem:shell} on $[m+1,2m]$ gives the
pair bound. For a low $i$ and a high $P_j$, the first high on the path is
a neighbor $U$. If $U=P_j$, use adjacency. Otherwise the remaining
whole-cell high path has length at least $\thR(U,P_j)>\thR(i,P_j)$.
For two distinct lows $i,j$, let $U$ and $V$ be the first and last highs.
If $U=V$, the simple path has two edges and its first gap alone is at
least $\thR(i,U)>\thR(i,j)$. If $U\ne V$, the middle high path is at
least $\thR(U,V)>\thR(i,j)$. These cases exhaust all pairs and both arcs,
including paths crossing the cyclic cell $P_m\to1\to P_1$.
For $m=2$ each directed high path is one whole cell; no third high is
required. The low-high and low-low arguments remain valid.

When $S_P(R)\le2\pi$, take $x_i=a_i$, $y_i=d_i-a_i$ and add the
remaining $2\pi-S_P(R)$ to any single gap. Each directed path either
increases or stays fixed; both bounds survive and closure holds. This
also proves sufficiency of every vector in \eqref{eq:gapset}.
Each of $a_i+b_i,c_i$ is positive, continuous and strictly decreasing
in $R$; their maximum is strictly decreasing by choosing an attaining
branch at the larger radius. Moreover $S_P(0+)=2\pi m$ and
$S_P(\infty)=0$. The intermediate value theorem and monotonicity give
the unique root and its feasible placement.
\end{proof}

\section{Recovery of General Reflected Blocks}\label{sec:transfer}
Fix $0\le\alpha<1$ and nonnegative lengths $\ell_1,\ldots,\ell_k$,
$k\ge1$, with $T=\sum_j\ell_j<1-\alpha$. Put $t_0=0$,
$t_j=\sum_{h\le j}\ell_h$, $A=1+\alpha$, and
$h_\alpha(t)=1+\{t+\alpha\}$. Define a probability measure by
\begin{equation}\label{eq:mu}
\begin{split}
 \int F\dd\mu_{\alpha;\ell}
 ={}&\frac12\sum_{j=1}^k\int_{t_{j-1}}^{t_j}
 \bigl[F(t,A+t,A+t_{j-1}+t_j-t)\\
 &\hspace{38mm}+F(t,A+t_{j-1}+t_j-t,A+t)\bigr]\dd t\\
 &+\int_T^1F(t,h_\alpha(t),h_\alpha(t))\dd t.
\end{split}
\end{equation}
Each high marginal and the low marginal are uniform on their respective
intervals. Each slab reflection is a measure-preserving involution, and
the unchanged shift is a bijection modulo its wrap. Zero lengths and
endpoint assignments do not affect the measure. Adjacent slabs are
reflected separately, never as their union.

\begin{theorem}[Quantitative Full Transfer]\label{thm:transfer}
Let $I=\int g\dd\mu_{\alpha;\ell}$ and
$M_0=\max\{2,\lceil2/(1-\alpha-T)\rceil\}$.
There are genuine orders on $1,\ldots,2m$ with full minimum $\rho_m$
for every $m\ge M_0$, and, for $m\ge\max\{M_0,2048\}$,
\begin{equation}\label{eq:radiuserror}
 \left|\frac{\rho_m}{(2m)^2}-\frac I{4\pi}\right|
 \le\frac1{4\pi}\left(\frac{6k^2+28k+1060}{m}
                         +\frac{16384}{3m^2}\right).
\end{equation}
Consequently $\limsup_n\Rstar(n)/n^2\le C(\mu_{\alpha;\ell})$.
No chord-dominance assumption on any slab is required.
\end{theorem}
\begin{proof}
Define exact integers
\[
 s=\lfloor\alpha m\rfloor,\quad l_j=2\lfloor\ell_jm/2\rfloor,
 \quad a_0=0,\quad a_j=\sum_{h\le j}l_h,\quad z=a_k,
 \quad r=m-s,
\]
and $H(u)=m+1+((u+s-1)\bmod m)$. For every $i\in\{1,\ldots,m\}$,
put $J(i)=2a_{j-1}+l_j+2-i$ if $i$ is even and belongs to a block
$(a_{j-1},a_j]$, and put $J(i)=i$ otherwise. In particular, $J$ fixes
every index outside the inversion blocks.
Set $P_i=H(J(i))$, $P_0=P_m$. Each block reverses its even ranks and
fixes its odd ranks, giving a disjoint involution; $H$ is a bijection.
Lengths zero and two cause no exception. Since
$s+z\le m(\alpha+T)$, $r-z\ge2$. The blocks and their exits precede
the high wrap.

For completeness the cells have the following exhaustive inventory.
A nonempty block $(a,a+l]$ has interior indices $a+2,\ldots,a+l$.
The predecessor/current pair there is
\[
 \begin{cases}
 (m+s+i-1,m+s+2a+l+2-i),&i\text{ even},\\
 (m+s+2a+l+3-i,m+s+i),&i\text{ odd}.
 \end{cases}
\]
The exceptional set is
$\{1,r,r+1\}\cup\{a_j+1:l_j>0\}$, intersected with
$\{1,\ldots,m\}$, with duplicates counted once. It has at most $k+3$
members. Cell $1$ is $(H(m),H(1))$; the exit from $(a,a+l]$ is
$(H(a+2),H(a+l+1))$. It is also the next nonempty block's entry even
if zero lengths intervene. Cell $r$ is $(2m-1,2m)$, and cell $r+1$,
when present, is $(2m,m+1)$. All other cells are ordinary
$(H(i-1),H(i))$. The interiors exclude every shared entry index, and
the wrap lies outside them, proving disjointness and completeness.

Let $\mu_m=m^{-1}\sum_i\delta_{(i/m,P_{i-1}/m,P_i/m)}$.
For a continuous, possibly nonsymmetric $F$, write
$B_F=\norm F_\infty$ and $\omega_F$ for its max-norm modulus.
Replace $\alpha,\ell_j$ in \eqref{eq:mu} by $s/m,l_j/m$ to obtain
$\bar\mu_m$. On each double panel
$[(a+2h-2)/m,(a+2h)/m]$, assign half the first-orientation integral
to its even cell and half the reverse-orientation integral to its odd
cell. Each mass is $1/m$. Interior coordinate errors are at most $4/m$;
ordinary diagonal panels have errors at most $1/m$. Exceptional panel
assignments each cost at most $2B_F/m$, so
\[
 \left|\int F\dd(\mu_m-\bar\mu_m)\right|
 \le\omega_F(4/m)+2(k+3)B_F/m.
\]
The floor errors are less than $1/m$ for the shift and $2/m$ for each
length. The intervals between $a_j/m$ and $t_j$, together with the wrap
interval, have total length at most $[k(k+1)+1]/m$. This bound allows
overlap and zero slabs. Off these intervals the same slab and orientation
apply; the reflected coordinate error in slab $j$ is less than
$(4j-1)/m$ and the other is less than $1/m$. Thus
\begin{equation}\label{eq:weakerror}
 \left|\int F\dd(\mu_m-\mu)\right|
 \le\omega_F(4/m)+\omega_F((4k-1)/m)
       +\frac{(2k^2+4k+8)B_F}{m}.
\end{equation}

The complete maximum $g$ obeys $1\le g<3$ and is $4$-Lipschitz.
For $t\le1/4$ the chord dominates because
$\sqrt{t/X}+\sqrt{t/Y}\le2\sqrt t\le1$. Above $1/4$, the sum of
absolute partial derivatives of the chain is at most $2\sqrt2+1<4$,
and that of the chord is less than two. The maximum of these Lipschitz
functions has the same bound; splitting at $t=1/4$ proves it globally.
Putting $G_m=\int g\dd\mu_m$, \eqref{eq:weakerror} yields
\begin{equation}\label{eq:scoreerror}
 |G_m-I|\le(6k^2+28k+36)/m.
\end{equation}
This estimate retains every finite maximum and seam even if its branch
differs from its continuum counterpart.

By Theorem~\ref{thm:cells}, the actual full minimum is the root of
$S_m(R)=\sum_i\max\{\thR(P_{i-1},i)+\thR(i,P_i),
\thR(P_{i-1},P_i)\}=2\pi$. For $R=4cm^2$, $c\ge1/32$, $m\ge32$,
and $1\le u,v\le2m$, put $q=\sqrt{uv}/R$. Then $q\le1/(2cm)\le1/2$.
The relative denominator change satisfies
\[
 0\le1-\frac1{\sqrt{(1+u/R)(1+v/R)}}\le\frac{u+v}{2R}.
\]
Also $0\le\arcsin t-t\le t^3/3$ for $0\le t\le1/2$,
by integrating $1/\sqrt{1-t^2}-1\le t^2$ on that interval.
It follows that
\[
 \left|\thR(u,v)-\frac{2\sqrt{uv}}R\right|
 \le\frac{512}{m^2}+\frac{8192}{3m^3},\qquad
 \left|S_m(4cm^2)-\frac{G_m}{2c}\right|
 \le E_m:=\frac{1024}m+\frac{16384}{3m^2}.
\]
The last step uses the one-Lipschitz property of the scalar maximum and
two angles per chain. For $m\ge2048$, $E_m<1$; at $c=1/32$ the score
exceeds $15>2\pi$, while at $c=1/2$ it is below $4<2\pi$.
Thus the actual normalized root $c_m=\rho_m/(2m)^2$ is first bracketed
in $(1/32,1/2)$. Substitution then gives
$|4\pi c_m-G_m|\le2c_m E_m\le E_m$.
Together with \eqref{eq:scoreerror} this is \eqref{eq:radiuserror}.
The root placement is all-pairs feasible. Deleting only label $2m$
keeps exactly $1,\ldots,2m-1$ and all their constraints; the normalization
factor $(2m/(2m-1))^2$ tends to one. This proves the global limsup
bound for both parities without identifying an odd fixed-order minimum.
\end{proof}

The same construction allows countably many lengths with total
$T<1-\alpha$. Truncation after $k$ blocks and reversion of the tail to
the diagonal changes any $F$ integral by at most
$2B_F\sum_{j>k}\ell_j$. For $g$ the bound is six times that tail length.
Apply the finite theorem and then let $k\to\infty$. A single recovery
sequence also results by choosing increasing thresholds $M_k$ growing
faster than $k^3$, each satisfying the common geometric gate, and using the largest $k\le m$ with $M_k\le m$.
Both \eqref{eq:weakerror} for Lipschitz tests and the omitted tail vanish;
density of Lipschitz functions on compact $\mathcal D$ gives continuous-test
recovery. This supplies countable upper constructions, not an optimum
over partitions.

\section{The Exact Normalized Prefix Parameter}\label{sec:x}
For $0\le u\le x\le1$, define
\begin{align*}
 c(u,x)&=\sqrt{(1+u)(1+x-u)},\\
 k(u,x)&=\sqrt u(\sqrt{1+u}+\sqrt{1+x-u}),\\
 d(u)&=\max\{1+u,2\sqrt{u(1+u)}\}.
\end{align*}
\begin{equation}\label{eq:E}
 E(x)=\int_0^x[\max\{c(u,x),k(u,x)\}-d(u)]\dd u.
\end{equation}
\begin{theorem}\label{thm:x}
The function $E$ has a unique global minimum $x_*$ on $[0,1]$, and
$719/2500<x_*<2877/10000$. These are the exact bounds
$0.2876<x_*<0.2877$.
\end{theorem}
\begin{proof}
The ratio $v(u,x)=k/c=\sqrt{u/(1+u)}+\sqrt{u/(1+x-u)}$
strictly increases in $u$. Its endpoint equals one at the unique
$\tau\in(1/4,1/3)$ solving
$\sqrt{\tau/(1+\tau)}+\sqrt\tau=1$.
Set $z(x)=x$ for $x\le\tau$; for $x>\tau$ let $z\in(0,x)$ be
the unique unsquared root $v(z,x)=1$. The block is chord up to $z$,
then chain. The diagonal switches at $u=1/3$.
For $x>\tau$, with $y=1+x-z$ and $r=\sqrt{z/(1+z)}$,
\begin{equation}\label{eq:zprime}
 z'=\frac{z}{(y/(1+z))^{3/2}+1+x}
    =\frac{(1-r)^2}{2-r+2r^2}>0.
\end{equation}
The second equality uses $y=z/(1-r)^2$ and $z=r^2/(1-r^2)$.
All denominators are positive.

Let $\Phi=E'$. Cost equality at the switch cancels its first moving
boundary term. The endpoint ties at $x=\tau,1/3$ also agree, so $E$
is $C^1$ there. On the chord interval,
\begin{gather*}
 \Phi(x)=\sqrt{1+x}+(1+x/2)\arcsin\frac{x}{2+x}-(1+x),\\
 \Phi(0)=\Phi'(0)=0,\qquad
 \Phi''(x)=-\frac1{2(2+x)(1+x)^{3/2}}<0.
\end{gather*}
Thus $\Phi<0$ on $(0,\tau]$. On $x>\tau$, put
\[
 I_1=\int_0^z\sqrt{\frac{1+u}{1+x-u}}\dd u,\qquad
 I_2=\int_z^x\sqrt{\frac u{1+x-u}}\dd u,\qquad
 p=\sqrt x(\sqrt{1+x}+1).
\]
Differentiation yields
\begin{equation}\label{eq:Phi}
 \Phi=p-d(x)+(I_1+I_2)/2,\qquad
 \Phi'=p'-d'(x)+\sqrt x/2-J/4+Q,
\end{equation}
where
\[
 J=\int_0^z\frac{\sqrt{1+u}}{(1+x-u)^{3/2}}\dd u
   +\int_z^x\frac{\sqrt u}{(1+x-u)^{3/2}}\dd u,\quad
 Q=\frac{z'(\sqrt{1+z}-\sqrt z)}{2\sqrt{1+x-z}}>0.
\]
For $\tau<x<1/3$, $d'=1$, $J\le x\sqrt{1+x}<7/18$ and
$p'>1/\sqrt x\ge\sqrt3>5/3$. Hence $\Phi'>41/72>0$.
At $x=1/3$, $v(2/7,x)<1$: the pre-square residual is $50/99>0$
and its square minus $4(2/9)(3/11)$ is $124/9801>0$.
Thus $z>2/7$, and
\[
 \Phi(1/3)>\frac1{\sqrt3}-\frac23+\frac{\sqrt3}{14}
 =\frac{17\sqrt3}{42}-\frac23>0
\]
by $17^2\cdot3>28^2$. There is exactly one zero $x_*$ in
$(\tau,1/3)$, with strict decrease then increase of $E$ on $[0,1/3]$.

Global minimality needs a separate tail argument; convexity must not be
extrapolated across $1/3$. On $1/3<x\le1$, \eqref{eq:Phi} becomes
\[
 \Phi'=-M-J/4+Q,\qquad
 M=\frac{1+2x-(1+x)^{3/2}}{2\sqrt{x(1+x)}}.
\]
Since $z>2/7$, $r>\sqrt{2/9}>8/17$. The polynomial
$5-13r+5r^2$ decreases on $(0,1)$ and is $-3/289$ at $8/17$;
\eqref{eq:zprime} gives $z'<1/7$. Also
$\sqrt{1+z}-\sqrt z<2/3$, so $Q<1/21$.
Write $q=\sqrt{1+x}\in[2/\sqrt3,\sqrt2]$. Then
\[
 M=\sqrt{\frac{q-1}{q+1}}\frac{1+q-q^2}{2q}.
\]
For $q\le4/3$ the factors exceed $1/4$ and are at least $5/24$,
respectively, so $M>5/96>1/20$. For $q\ge4/3$ they are at least
$1/\sqrt7>3/8$ and $(2-\sqrt2)/4>1/7$, so again $M>1/20$.
These comparisons follow by positive squares, using
$2/\sqrt3>17/15$ and $\sqrt2<10/7$. Therefore
$\Phi'<-1/420$ throughout the tail.
The rational gates in Section~\ref{sec:gates} give $E(1)>3/250>0$.
Strict concavity makes $E$ at least its endpoint chord on $[1/3,1]$.
Both endpoint values exceed $E(x_*)<0$, proving unique global
minimality. The same gates give opposite signs at $719/2500$ and
$2877/10000$, both above $\tau$, locating this unique root. They also
give $\Phi(4/5)<0$; thus a later local maximum exists, and does not
invalidate the endpoint-chord argument.
\end{proof}

\section{The Exact Shift Parameter}\label{sec:alpha}
For $0\le\alpha\le1/2$ let $A=1+\alpha$, $a=A/3$, $b=1-\alpha$,
and
\[
 K(\alpha)=\int_0^1\max\{\sqrt{t h_\alpha(t)},h_\alpha(t)/2\}\dd t,
 \qquad e=E(x_*),\qquad F(\alpha)=K(\alpha)+A^2e/2.
\]
The full prefix cost at length $Ax_*$ is $F(\alpha)/(2\pi)$.
Indeed $Ax_*<a\le b$, the replaced diagonal is chord, and scaling
$t=Au$ gives its increment $A^2E(x_*)$. Moreover
$b-Ax_*>1/2-(3/2)(2877/10000)=1369/20000>1/15$,
so the reflection stays before the wrap even at the closed endpoints.

\begin{theorem}\label{thm:alpha}
There is exactly one zero $\al$ of
$F'(\alpha)=K'(\alpha)+(1+\alpha)e$ on $[0,1/2]$.
It is the unique minimum of this fixed-$x_*$ prefix family, and
$1093/10000<\al<10931/100000$.
\end{theorem}
\begin{proof}
The split integral is
\[
 K=\int_0^a(A+t)/2\dd t+\int_a^b\sqrt{t(A+t)}\dd t
                       +\int_b^1\sqrt{t(\alpha+t)}\dd t.
\]
Differentiation, including the moving wrap, gives
\begin{equation}\label{eq:D}
 D(\alpha):=K'(\alpha)
 =\frac a2+\frac12\int_a^b\sqrt{\frac t{t+A}}\dd t
 +\frac12\int_b^1\sqrt{\frac t{t+\alpha}}\dd t
 -(\sqrt2-1)\sqrt b.
\end{equation}
For $0<x<1/3$, centering the chord at $M=1+x/2$ gives
$E(x)\ge-\int_{-x/2}^{x/2}u^2/M\dd u=-x^3/(12M)>-1/324$.
At $x=1/4$, which is strictly chord, rationalization gives
$-E(1/4)>1/1728$. Theorem~\ref{thm:x} implies
$-1/324<e<-1/1728$.
A second differentiation of \eqref{eq:D} gives
\[
 K''=\frac1{12}+\frac{(1-1/\sqrt2)\sqrt b}{2}
       +\frac{\sqrt2-1}{2\sqrt b}-\frac{T_1}{4},
 \quad
 T_1=\int_a^b\frac{\sqrt t}{(t+A)^{3/2}}\dd t
       +\int_b^1\frac{\sqrt t}{(t+\alpha)^{3/2}}\dd t.
\]
Both denominator bases are at least one; $T_1\le1-a\le2/3$.
Using $1/2\le b\le1$ and $\sqrt2>7/5$ yields $K''>7/60$,
so $F''>7/60-1/324=46/405>1/9$ on the whole domain.

At zero,
$D(0)=5/6-\sqrt2/2-\tfrac12\log((1+\sqrt2)/\sqrt3)$.
For $z>1$, $\log z>2(z-1)/(z+1)$ by differentiation. Since the
argument exceeds $48/35$, using $\sqrt3<7/4$ gives
$F'(0)<5/6-7/10-13/83=-29/1245<0$.
At $1/2$ the surviving integrand in \eqref{eq:D} is at least
$1/\sqrt2$, so $D(1/2)\ge5/(4\sqrt2)-3/4>1/8$ and
$F'(1/2)>1/8-(3/2)/324=13/108>0$.
Strict convexity proves uniqueness and minimality.

For the tighter location set $x_0=5753/20000$, the midpoint of the
$x_*$ bracket. On $(\tau,1/3)$, \eqref{eq:Phi} gives
$E''<8/3-1+1/3+1/8=17/8<3$: here $p'<8/3$,
$z'<x/(1+x)<1/4$, $Q<1/8$, and $J\ge0$.
Taylor's theorem at the stationary minimum gives
\begin{equation}\label{eq:ebox}
 E(x_0)-3/800000000<e\le E(x_0).
\end{equation}
The directed gates in Section~\ref{sec:gates}, evaluating only this
rational argument and the two rational $\alpha$ endpoints, give
$F'(1093/10000)<0<F'(10931/100000)$. These isolate the same unique zero.
\end{proof}

\section{A Common Width Lemma and the Next Two Parameters}\label{sec:widths}
Fix $A>1$ and $A/4<u<A/3$. Put $a=A/3$, $B=A+u$, and $h=a-u$,
so $B>5A/4$, $0<h<A/12$, $h/B<1/15$. Reflect a diagonal slab
$[u,u+d]$, $0\le d\le h$, separately from all preceding slabs. Its
unnormalized complete-cost increment is exactly
\begin{equation}\label{eq:widthcost}
\begin{split}
 D_{A,u}(d)=\int_0^d\bigl[\max\{&\sqrt{(B+s)(B+d-s)},\\
 &\sqrt{u+s}(\sqrt{B+s}+\sqrt{B+d-s})\}-(B+s)\bigr]\dd s.
\end{split}
\end{equation}
The removed diagonal is chord because $A-3(u+s)\ge0$; its only possible
tie is $s=d=h$. Reflection preserves each marginal as in
\eqref{eq:mu}, and unchanged slabs cancel with their full maxima.

\begin{lemma}[Unique Mixed Width]\label{lem:width}
The function $D_{A,u}$ has a unique minimum $d_*$ on $[0,h]$.
Define $\tau_u$ as the unique unsquared root in $(0,h)$ of
\[
 W(d)=\sqrt{\frac{u+d}{B+d}}+\sqrt{\frac{u+d}B}=1.
\]
Then $\tau_u<d_*<\tau_u+\tau_u^2/(8B)$, and $d_*$ is the unique
mixed-branch stationary point. The cost decreases strictly up to $d_*$
and increases strictly afterwards.
\end{lemma}
\begin{proof}
Set $t=u+s$, $X=B+s$, $Y=B+d-s$, $c=\sqrt{XY}$,
$k=\sqrt t(\sqrt X+\sqrt Y)$, and $V=k/c$.
Then
\[
 V_s=\frac{A/X^{3/2}+(Y+t)/Y^{3/2}}{2\sqrt t}>0.
\]
For $d>0$, $V(0,d)<1$ and
$V(d/2,d)=2\sqrt{(u+d/2)/(A+u+d/2)}<1$.
The endpoint $W$ strictly increases, with $W(0)<1$ and
$W(h)=1/2+\sqrt{a/B}>1$. This proves the asserted entry width.
For $d\le\tau_u$ set $z=d$. For $d>\tau_u$ there is a unique
$z\in(d/2,d)$ with $V(z,d)=1$. The cost is
\[
 D_{A,u}=\int_0^z c\dd s+\int_z^d k\dd s-Bd-d^2/2.
\]
Let $\Psi=D_{A,u}'$. In the all-chord regime,
\[
 \Psi(d)=B\phi(d/B),\qquad
 \phi(q)=\sqrt{1+q}+(1+q/2)\arcsin\frac q{2+q}-(1+q).
\]
Here $\phi(0)=\phi'(0)=0$ and
$-1/4<\phi''(q)=-[2(2+q)(1+q)^{3/2}]^{-1}<0$ for $q>0$.
Therefore $-\tau_u^2/(8B)<\Psi(\tau_u)<0$ and the positive-width
chord cost is strictly decreasing.

In the mixed regime put $q=u+d$,
$p=\sqrt q(\sqrt{B+d}+\sqrt B)$,
$I_c=\int_0^z\sqrt{X/Y}\dd s$ and $I_k=\int_z^d\sqrt{t/Y}\dd s$.
Leibniz differentiation gives
\begin{equation}\label{eq:Psi}
 \Psi=p-(B+d)+(I_c+I_k)/2,\qquad
 z'=\frac{t}{A(Y/X)^{3/2}+Y+t}\in(0,1/2),
\end{equation}
where $t,X,Y$ in $z'$ are evaluated at $s=z$.
The first moving-boundary term cancels by $c(z)=k(z)$.
The bound on $z'$ follows from $Y-t=A+d-2z>A-d>0$.
The cost and first derivative are continuous at entry: the chain interval
shrinks to zero and the endpoint costs agree. Equivalently put $s=dr$
on $[0,1]$; positive radicands bound the difference quotients and the tie
set has measure zero, permitting dominated differentiation. This also
justifies the one-sided derivative at $h$.

The second derivative retains a nonzero moving-switch term:
\begin{equation}\label{eq:Psiprime}
 \Psi'=p'-1+\tfrac12\sqrt{q/B}-J/4+Q,
 \quad p'=\frac{A+2q}{2\sqrt{q(A+q)}}+\frac{\sqrt B}{2\sqrt q},
\end{equation}
where
\[
 J=\int_0^z\frac{\sqrt X}{Y^{3/2}}\dd s+
       \int_z^d\frac{\sqrt t}{Y^{3/2}}\dd s,
 \qquad Q=\frac{z'(\sqrt{B+z}-\sqrt{u+z})}{2\sqrt{B+d-z}}>0.
\]
The first term in $p'$ decreases in $q$ and equals $5/4$ at $q=a$.
Thus $p'>5/4+\sqrt{15}/4>17/8$.
Also $0<J\le d\sqrt{B+d}/B^{3/2}<2/15$, giving
$\Psi'>131/120>1$ everywhere in the mixed regime.
Cost equality cannot be used to drop $Q$, since the derivative integrands
at the switch do not coincide.

At $h$, $B=4a-h$, $u=a-h>3a/4$, and
\[
 \Psi(h)=\sqrt{aB}-2a+(I_c+I_k)/2,\qquad
 2a-\sqrt{aB}=\frac h{2+\sqrt{B/a}}<2h/7.
\]
On this slab $15a/4<X,Y\le4a$ and $3a/4<t\le a$, so
$\sqrt{X/Y}>9/10$ and $\sqrt{t/Y}>2/5$.
Since $z(h)>h/2$, $(I_c+I_k)/2>13h/40$, whence
$\Psi(h)>11h/280>0$. Continuity and strict increase of $\Psi$ in the
mixed regime give its unique zero $d_*$. The chord sign excludes earlier
stationary points, proving the unique global minimum including both
endpoints. Finally integrating $\Psi'>1$ from $\tau_u$ to $d_*$ gives
$d_*-\tau_u<-\Psi(\tau_u)<\tau_u^2/(8B)$.
\end{proof}

We now fix $A=1+\al$ and $\lambda=Ax_*$. Define
\begin{equation}\label{eq:epsdefinition}
 \epsb=\mathop{\rm argmin}_{0\le d\le A/3-\lambda}D_{A,\lambda}(d).
\end{equation}
Lemma~\ref{lem:width} proves this is a unique mixed stationary minimum.
The rational gates below give
\[
 43/1000<\tau_\lambda<87/2000,
 \quad 43/1000<\epsb<87/2000+(87/2000)^2/8<7/160.
\]
In particular $v=\lambda+\epsb$ remains in $(A/4,A/3)$.
Define in turn
\begin{equation}\label{eq:Deltadefinition}
 \Ds=\mathop{\rm argmin}_{0\le d\le A/3-v}D_{A,v}(d).
\end{equation}
The same lemma proves its uniqueness and identifies its stationarity
with $4\pi C_3'(\Ds)=0$. The second pair of rational gates gives
\begin{equation}\label{eq:Deltabracket}
\begin{gathered}
 29/5000<\tau_v<67/10000,\\
 29/5000<\Ds<67/10000+(67/10000)^2/8<27/4000<A/3-v.
\end{gathered}
\end{equation}
These definitions retain all mixed branches. They do not use a chord-only
stationary equation or a rounded parameter.

\section{Rational Gates and Their Reproduction}\label{sec:gates}
This section supplies the finite arithmetic premises used above. Every
checker in the accompanying \texttt{checks} directory is self-contained;
none imports production code, the independent geometric verifier, another
checker, or task records. The first two have an optional numerical mode;
use \texttt{--exact-only} for the theorem gates. Run with ordinary Python
assertions enabled (do not use \texttt{-O}).

\subsection{Prefix Gates}
The file \texttt{upper\_lambda.py} uses interval endpoints on the grid
$10^{-50}$. Each arithmetic operation rounds outwards, division across
zero is rejected, and $\sqrt{n/10^{50}}$ is enclosed using the integer
square root of $n10^{50}$. Two half-angle reductions express
$\arctan y=4\arctan q$ with $|q|<1$. The first 64 alternating terms
are widened by $|q|^{129}/129$ before multiplying by four.
For $\log y$, take two square roots, set
$q=(y^{1/4}-1)/(y^{1/4}+1)$, and use the 64-term odd series for
$\log y=8\sum_{j\ge0}q^{2j+1}/(2j+1)$, with the corresponding
tail bounded by $8|q|^{129}/[129(1-|q|^2)]$.
Every operation in those formulas is itself enclosed.

At rational $x>\tau$, 100 dyadic bisections of $[0,x]$ isolate the
unsquared equation $v(z,x)=1$. An unseparated sign aborts the check.
The exact primitives
\[
 \mathcal F(t,c)=\frac{(2t+c)\sqrt{t(t+c)}
     -c^2\log((\sqrt t+\sqrt{t+c})/\sqrt c)}4,
\]
\[
 \mathcal J(v,M)=\frac{v\sqrt{M^2-v^2}
                +M^2\arctan(v/\sqrt{M^2-v^2})}2,
 \qquad
 \mathcal H(c,t)=c\arctan\sqrt{t/(c-t)}-\sqrt{t(c-t)}
\]
have derivatives $\sqrt{t(t+c)}$, $\sqrt{M^2-v^2}$ and
$\sqrt{t/(c-t)}$, respectively. They evaluate the split integrals in
\eqref{eq:E} and \eqref{eq:Phi}. Explicitly the block integral is
\begin{align*}
 &\mathcal J(z-x/2,1+x/2)-\mathcal J(-x/2,1+x/2)\\
 &\quad+\mathcal F(x,1)-\mathcal F(z,1)\\
 &\quad+\mathcal J(x-(1+x)/2,(1+x)/2)
 -\mathcal J(z-(1+x)/2,(1+x)/2).
\end{align*}
The last four terms are omitted in the chord regime. Subtract
$x+x^2/2$ for $x\le1/3$, or
$7/18+2[\mathcal F(x,1)-\mathcal F(1/3,1)]$ otherwise.
For \eqref{eq:Phi},
$I_1=\mathcal H(2+x,1+z)-\mathcal H(2+x,1)$ and
$I_2=\mathcal H(1+x,x)-\mathcal H(1+x,z)$.
This gives the strict rational gates
\begin{center}
\begin{tabular}{lll}\toprule
Quantity & Lower Bound & Upper Bound\\\midrule
$\Phi(719/2500)$ & $-1/20000$ & $-1/25000$\\
$\Phi(2877/10000)$ & $1/10000$ & $11/100000$\\
$\Phi(4/5)$ & $-1/500$ & $-1/1000$\\
$E(1)$ & $3/250$ & $13/1000$\\\bottomrule
\end{tabular}
\end{center}
Direct directed square roots also show $719/2500>\tau$.
These four gates establish constant signs, not the continuum derivative
inequalities proved in Section~\ref{sec:x}.

\subsection{Shift Gates}
The file \texttt{upper\_alpha\_minimum.py} uses exact rational arithmetic
and directed roots on the grid $10^{-24}$, checking each endpoint by
squaring. It isolates $z(x_0)$ with 64 sign-separated bisections.
On $[0,z_l]$ it integrates $c$ and on $[z_h,x_0]$ it integrates $k$,
using 2048 panels each. Both integrands are concave: their second
derivatives are $-M^2/c^3$ for the chord and the sum of
$-1/[4(u(1+u))^{3/2}]$ and
$-(1+x_0)^2/[4(u(1+x_0-u))^{3/2}]$ for the chain.
Thus composite trapezoids are lower bounds and midpoints are upper bounds.
The omitted interval contributes between zero and $4(z_h-z_l)$.
Subtract the exact diagonal and apply \eqref{eq:ebox}.
For each $D$ integral in \eqref{eq:D}, use 256 panels and the concavity
\[
 \frac{\mathrm d^2}{\mathrm dt^2}\sqrt{\frac t{t+c}}
 =-\frac{c(4t+c)}{4t^{3/2}(t+c)^{5/2}}<0.
\]
The negative wrap term is bounded in the reversed direction.
The resulting outward intervals have common denominator $10^{12}$:
\begin{center}
\begin{tabular}{lrr}\toprule
Quantity & Lower Numerator & Upper Numerator\\\midrule
$E(x_0)$ & $-844268665$ & $-844268070$\\
$e$ & $-844272415$ & $-844268070$\\
$D(1093/10000)$ & $935124205$ & $935201790$\\
$F'(1093/10000)$ & $-1427184$ & $-1344781$\\
$D(10931/100000)$ & $938842182$ & $938919761$\\
$F'(10931/100000)$ & $2282349$ & $2364747$\\\bottomrule
\end{tabular}
\end{center}
No evaluation at an uncertified irrational root or floating quadrature
is used. The analytic convexity proof supplies uniqueness.

\subsection{Second and Third Width Gates}
Write $\alpha_L=1093/10000$, $\alpha_H=10931/100000$,
$x_L=719/2500$, $x_H=2877/10000$, and $A_L=1+\alpha_L$,
$A_H=1+\alpha_H$.
For the second width, put
$f(x,q)=\sqrt{(x+q)/(1+x+q)}+\sqrt{(x+q)/(1+x)}$.
It increases in $q$ and, for $q<1$, in $x$; the squared summands'
$x$ derivatives are $1/(1+x+q)^2$ and $(1-q)/(1+x)^2$.
Consequently
\[
 W(43/1000)<f(x_H,(43/1000)/A_L)<1,
 \quad W(87/2000)>f(x_L,(87/2000)/A_H)>1.
\]
For positive rational squared summands $p,q$ with residual $r=1-p-q>0$,
\begin{equation}\label{eq:squaregate}
 \operatorname{sign}(\sqrt p+\sqrt q-1)
 =\operatorname{sign}(4pq-r^2).
\end{equation}
The file \texttt{upper\_boundary.py} checks the positive residuals and
the respective exact signed differences
\[
 -\frac{342531413518344092144807428477}{441789962670730640526171361763841},
 \qquad
 \frac{51974394530491111127115563}{253438432681061908372345318521}.
\]
Also $B>1$ and
$7/160-87/2000-(87/2000)^2/8=431/32000000>0$,
which proves the sharpened upper bound for $\epsb$.

Put $e_L=43/1000$, $e_U=7/160$ and
\[
 V_L=A_Lx_L+e_L=9050867/25000000,\qquad
 V_U=A_Hx_H+e_U=362898487/10^9.
\]
Thus $V_L<v<V_U$. For
$W(A,v,d)=\sqrt{(v+d)/(A+v+d)}+\sqrt{(v+d)/(A+v)}$,
both squared summands decrease in $A$ and increase in $v$ when $d<A$.
The file \texttt{upper\_mixed\_width.py} checks
\[
 W(A_L,V_U,29/5000)<1,\qquad W(A_H,V_L,67/10000)>1.
\]
Their signed differences \eqref{eq:squaregate}, with positive residuals,
are respectively
\[
 -\frac{963745517404602389188679409068317}{4734571866017504812540302482915910561},
 \qquad
 \frac{34763358497298510706339907}{205268116362886684743388903969}.
\]
Finally
\[
 \frac{27}{4000}-\frac{67}{10000}-\frac{(67/10000)^2}{8}
 =\frac{35511}{800000000}>0,\qquad
 \frac{A_L}3-V_U-\frac{27}{4000}
 =\frac{354539}{3000000000}>0.
\]
These show that all corner evaluations are admissible and prove
\eqref{eq:Deltabracket}. The same margins give $A/4<v<A/3$ as required
when Lemma~\ref{lem:width} is applied a second time.

\section{The Fourth Block and the Final Upper Endpoint}\label{sec:fourth}
Fix all four implicit parameters above and write $B_4=A+w$,
$\eta_0=1/20000$, $M=B_4+\eta/2$.
The independent rational enclosing corner
$W_U=V_U+27/4000=369648487/10^9$ gives, for $0<\eta\le\eta_0$,
\begin{align}
 A-3w&>354539/10^9, & A-3w-4\eta&>154539/10^9,\nonumber\\
 A/3-w-\eta&>204539/(3\cdot10^9),
 &1-\al-w-\eta&>520991513/10^9,\label{eq:fourthgates}\\
 3/2-M&>21016513/10^9.&&\nonumber
\end{align}
Thus the fourth slab lies strictly before the diagonal cutoff and wrap.
Only this slab changes. For every continuous, possibly nonsymmetric $F$,
its measure increment is
\begin{multline*}
 \frac12\int_w^{w+\eta}
 [F(t,A+t,A+2w+\eta-t)+F(t,A+2w+\eta-t,A+t)\\
 -2F(t,A+t,A+t)]\dd t.
\end{multline*}
This preserves positivity and each marginal. The shared endpoint has zero
mass; the third and fourth reflections remain separate. No stationarity
of earlier blocks is used to cancel their unchanged full costs.

At $t=w+s$, $X=B_4+s$, $Y=B_4+\eta-s$, $0\le s\le\eta$,
\eqref{eq:fourthgates} implies
\[
 \frac{\sqrt t(\sqrt X+\sqrt Y)}{\sqrt{XY}}
 \le2\sqrt{(w+\eta)/B_4}<1,
 \qquad X-4t=A-3w-3s>0.
\]
Both the inserted and removed full maxima are strictly chord throughout
the closed slab. The unnormalized increment is therefore
\begin{align}
 4\pi[C_4(\eta)-C_3(\Ds)]
 &=\int_0^\eta[\sqrt{(B_4+s)(B_4+\eta-s)}-(B_4+s)]\dd s\nonumber\\
 &=-\int_{-\eta/2}^{\eta/2}
       \frac{z^2}{M+\sqrt{M^2-z^2}}\dd z.\label{eq:decrement}
\end{align}
The radicand is at least $B_4(B_4+\eta)>0$. Equivalently this is
\[ (\eta/2)\sqrt{B_4(B_4+\eta)}+M^2\arcsin(\eta/(2B_4+\eta))-\eta M. \]
Since
\[
 2B_4\le M+\sqrt{M^2-z^2}\le2M,\qquad
 \int_{-\eta/2}^{\eta/2}z^2\dd z=\eta^3/12,
\]
we obtain, with $J_4=C_4-C_3(\Ds)$,
\[
 -\frac{\eta^3}{96\pi B_4}\le J_4(\eta)
 \le-\frac{\eta^3}{96\pi M}< -\frac{\eta^3}{576}<0,
\]
using $M<3/2$ and $\pi<4$. Furthermore
\[
 J_4(\eta)=-\frac{\eta^3}{96\pi B_4}+\mathcal R_4(\eta),
 \qquad0\le\mathcal R_4(\eta)\le\frac{\eta^4}{192\pi B_4^2}.
\]
Indeed the remainder is at most $\eta^3(1/B_4-1/M)/(96\pi)\le\eta^4/(192\pi B_4^2)$. This is a signed finite-interval remainder and proves the first nonzero
unilateral term is cubic. It is not a negative first variation.
At $\eta_0$, the strict saving is
$\eta_0^3/576=1/4608000000000000$.

The four lengths sum to $T=w+\eta_0<1-\al$, so
Theorem~\ref{thm:transfer} applies with $k=4$. Its radius error is
$[1268/m+16384/(3m^2)]/(4\pi)$. The actual full roots and the odd
placements obtained by deletion prove \eqref{eq:target}. This establishes
the upper construction independently of any finite-size optimization.
It does not establish optimality of four blocks, of the marginal
relaxation, or of the global geometric problem.

\subsection{Independent Finite Check of the Fourth Saving}
The file \texttt{upper\_fourth\_block.py} verifies
\eqref{eq:fourthgates} with fractions and uses a 16-panel midpoint upper
enclosure of the raw inserted full maximum, subtracting the exact
diagonal integral. The branch gates above justify concavity in $s$.
For the chord difference its derivative in $B_4$ is
\[
 \frac{2B_4+\eta}{2\sqrt{(B_4+s)(B_4+\eta-s)}}-1\ge0,
\]
since the numerator's square minus four times the radicand is
$(\eta-2s)^2$. Therefore the upper enclosing $B_4$ corner gives an
upper cost bound. Directed roots on a $10^{-40}$ grid make this an exact
rational integral enclosure, independent of the antiderivative and the
rationalized form in \eqref{eq:decrement}. The code also checks 16 corners,
448 polynomial reflection moments and 48 branch probes, as bounded
bookkeeping diagnostics. It rejects zero, negative, and cutoff widths
for the strict chord-slab gate. Those probes are not the proofs of
arbitrary-test marginal preservation, universal recovery, or the cubic
coefficient; the preceding analytic arguments supply those quantifiers.

\section{Scope, Provenance, and Reproduction}
This text repackages existing exact theorems and proved corollaries from
the seven research sources indexed in \texttt{UPPER\_DEPENDENCIES.md}.
That document maps each section and checker to its original source at
the frozen scientific snapshot. The two width arguments are written once
with an explicit generic start, and checked under the hypotheses of both
applications. No new mathematical coefficient, finite certificate, or
claim of external mathematical acceptance is introduced.

From the directory containing this source, run
\begin{verbatim}
python -I -S checks/upper_lambda.py --exact-only
python -I -S checks/upper_alpha_minimum.py --exact-only
python -I -S checks/upper_boundary.py
python -I -S checks/upper_mixed_width.py
python -I -S checks/upper_fourth_block.py
\end{verbatim}
The first four copied scripts use assertions, which must remain enabled.
Their additional historical comparison outputs are not proof premises
of this supplement. Optional floating diagnostics are numerical
observations only. None of these checks certifies a finite Ringmin global
optimum. The full all-pairs and infinite-parameter statements are proved
in the text, including both cyclic arcs, actual-label recovery, all seams,
the root scale before asymptotic substitution, and both parities.